\documentclass[11pt,twoside]{article}
\usepackage{amssymb}
\usepackage{latexsym}
\usepackage{graphicx}
\usepackage{amsmath}
\usepackage{hyperref,amsthm}
\usepackage{epsfig}
\usepackage{booktabs}
\usepackage{amsmath}

\numberwithin{equation}{section}
\normalsize \makeatletter
\newtheorem{thm}{Theorem}[section]
\newtheorem{lemma}[thm]{Lemma}
\newtheorem{pro}[thm]{Proposition}
\newtheorem{corollary}[thm]{Corollary}

\newtheorem{rem}[thm]{Remark}

\newcommand{\be}{\begin{equation}}
\newcommand{\ee}{\end{equation}}
\newcommand{\bea}{\begin{eqnarray*}}
\newcommand{\eea}{\end{eqnarray*}}

 \usepackage{geometry}

\begin{document}
\date{}
\title{\textbf{Global existence and blow-up for the variable coefficient Schr\"{o}dinger equations with a linear potential  \footnote{2020 Mathematics Subiect Classification: 35Q55, 35B44.
}}}
\author{Bowen Zheng\footnote{E-mail: \emph{bwen\_zj1516@126.com} (B. Zheng).},\quad \  Tohru Ozawa \footnote{Corresponding author: \emph{txozawa@waseda.jp} (T. Ozawa).
}
\\\emph{College of Sciences, China Jiliang University},\\ \emph{Hangzhou 310018, P. R. China},\\
\emph{Department of Applied Physics, Waseda University},\\ \emph{Tokyo 169-8555, Japan}} \maketitle

\begin{abstract}
In this paper, we study a class of variable coefficient Schr\"{o}dinger equations with a linear potential
\[i\partial_tu+\nabla\cdot(|x|^b\nabla u)-V(x)u=-|x|^c|u|^pu,\]
where $2-n<b\leq0,\ c\geq b-2$ and $0<\textbf{p}_c\leq(2-b)(p+2)$, where $\textbf{p}_c:=np-2c$.
In the radial or finite variance case, we firstly prove the global existence and blow-up below the ground state threshold for the mass-critical and inter-critical nonlinearities. Next, adopting the variational method of Ibrahim-Masmoudi-Nakanishi \cite{IMN}, we obtain a sufficient condition on the nonradial initial data, under which the global behavior of the general solution is established.

\end{abstract}

\section{Introduction}
In this paper, we consider the Cauchy problem for the variable coefficient Schr\"{o}dinger equation with a linear potential (INLS$_{b,V}$ for short)
\begin{equation}\label{a0}
\left\{
\begin{aligned}
 & i\partial_tu+\nabla\cdot(|x|^b\nabla u)-V(x)u=-|x|^c|u|^pu,\quad\ (x, t)\in\Bbb{R}^n\times[0, +\infty),  \\
 &u(x,0)=u_0(x),
\end{aligned}\right.
\end{equation}
where $u:\ \Bbb{R}^n\times[0, +\infty)\rightarrow\Bbb{C},\ 2-n<b<2,\ c\geq b-2,\ 0<\textbf{p}_c:=np-2c\leq(2-b)(p+2)$ and the potential $V(x)\geq0$. 
The INLS$_{b,V}$ appears in the study of standing waves problem for the following degenerate elliptic equation:
\begin{equation}\nonumber
\nabla\cdot(|x|^b\nabla \phi)-V(x)\phi=-|x|^c|\phi|^p\phi,
\end{equation}
of which the properties of nontrivial solution have attracted a lot of interests from the mathematical community (see e.g.  \cite{DLY}, \cite{G3}, \cite{G4}, \cite{M}-\cite{SW2}, \cite{ZYS} and references therein).

We define the functional adapted to the Schr\"{o}dinger operator
\[\mathcal{A}_{b, V}:=-\nabla\cdot(|x|^b\nabla)+V(x)\]
by
\[\|u\|_{\dot{H}_{b,V}^1}^2:=\langle\mathcal{A}_{b,V}u,\ u\rangle=\int(|x|^b|\nabla u|^2+V(x)|u|^2)dx.\]

The INLS$_{b}$ (case $V=0$) enjoys a scaling invariance under the scaling $u\mapsto u_\lambda$, where
\begin{equation}\label{k7}
u_\lambda(x, t):=\lambda^{\frac{2-b+c}{p}} u(\lambda x, \lambda^{2-b}t), \quad \ \lambda>0.
\end{equation}
A direct calculation gives
\begin{equation}\nonumber
\|u_\lambda(\cdot,0)\|_{\dot{H}^{s}}=\lambda^{s+\frac{2-b+c}{p}-\frac{n}{2}}\|u_0\|_{\dot{H}^{s}},
\end{equation}
which shows that \eqref{k7} leaves the the usual homogeneous Sobolev $\dot{H}^{s_c}$-norm of the initial data invariant with the critical Sobolev index
\[s_c:=\frac{n}{2}-\frac{2-b+c}{p}.\]
The mass-critical case corresponds to $s_c=0$ (\emph{equivalently $\textbf{p}_c=2(2-b)$});
The energy-critical case corresponds to $s_c=\frac{2-b}{2}$ (\emph{equivalently $\textbf{p}_c=(2-b)(p+2)$}). The intercritical case corresponds to $s_c\in(0, \frac{2-b}{2})$ (\emph{equivalently $\textbf{p}_c\in(2(2-b),\ (2-b)(p+2))$}). For later uses, it is convenient to introduce the following exponent:
\begin{equation}\label{k6}
\sigma:=\frac{2-b-2s_c}{2s_c}=\frac{(2-b)(p+2)-\textbf{p}_c}{\textbf{p}_c-4+2b}.
\end{equation}

The main purpose of the present paper is to study long time dynamics (global existence, finite or infinite time blow-up) of the radial and non-radial solutions to the INLS$_{b,V}$.
Since $a(x)=|x|^b$ with $b<0$ can be singular at $x=0$, the local well-posedness analysis of the INLS$_{b,V}$ is more subtle and intricate in several aspects, for example the energy space associated with $\mathcal{A}_{b, V}$ is no longer usual Sobolev space, and Strichartz estimate arguments fails. To our aim, we take the LWP in $W_b^{1,2}$ related to the INLS$_{b,V}$  as an assumption when necessary and build part of our conditional result upon it. This means that we always assume that for any $u_0\in W^{1,2}_b$, the INLS$_{b,V}$  admits a unique local solution $u$ in $C([0, T^\ast);\ W^{1,2}_b)$ with the maximal lifetime $0<T^\ast<+\infty$. If
\[\lim_{t\rightarrow T^\ast}\|\nabla u(t)\|_{b,2}=\infty,\]
we call such $u$ is the finite time (\emph{resp. infinite time}) blow-up solution when $T^\ast<+\infty$ (\emph{resp. $T^\ast=+\infty$}).
In addition, the solutions to INLS$_{b,V}$ formally conserve their mass and energy (see e.g. \cite{CH}), 
\begin{eqnarray}
&&M(u(t)):=\int|u|^2dx=M(u_0),\\\label{f5}
&&E_{b,V}(u(t)):=\frac{1}{2}\|\nabla u\|^2_{b,2}+\frac{1}{2}\int V(x)|u|^2dx-\frac{1}{p+2}\|u\|^{p+2}_{c,p+2}=E_{b,V}(u_0).\quad
\end{eqnarray}

Let us review some recent development for the INLS$_{b,V}$. When $b=0$ and $V=0$,  \eqref{a0} reduces to the inhomogeneous nonlinear Schr\"{o}dinger equation (INLS)
\[ i\partial_tu+\Delta u=-|x|^c|u|^pu.\]
In this case, the local well-posedness in $H^1$ of the solutions were established by Geneoud-Stuart \cite{GS4}, whose proof makes use of the energy method developed by Cazenave \cite{TC}. See also \cite{GC4} for other proof based on Strichartz estimates argument. In the mass-critical case, Genoud \cite{G5} showed that the INLS with $\max\{-2, -n\}<c<0$ is globally well-posedness in $H^1$ by assuming $u_0\in H^1$ and $\|u_0\|_2<\|\mathcal{Q}\|_2$, where $\mathcal{Q}$ is the unique positive radial solution to the elliptic equation
\[\Delta \mathcal{Q}-\mathcal{Q}+|x|^c|\mathcal{Q}|^p\mathcal{Q}=0.\]
This result was also extended by Farah \cite{F4} to the intercritical case, and he also showed that the solution blows up in finite time (see also \cite{D5}).

When $b=0$ and $V\neq0$, the global existence and blow-up for the INLS$_{V}$ equation
\[i\partial_tu+\Delta u-Vu=-|x|^c|u|^pu\]
were already extensively studied in recent years. Under the assumptions that
$V\in\mathcal{K}_0\cap L^{\frac{3}{2}}$ and $\|V_-\|_{\mathcal{K}}:=\sup_{x\in\Bbb{R}^3}\int\frac{|V_-(y)|}{|x-y|}dy<4\pi,$ Hong \cite{H1} studied the
INLS$_{V}$ and established the local well-posedness in $H^1$. In case $c=0$, Hamano-Ikeda \cite{HI} proved that the intercritical INLS$_{V}$ in three dimension with $V\geq0,\ L^{\frac{3}{2}}\ni x\cdot\nabla V\leq0$ is globally well-posed in $H^1$ if $u_0\in H^1$ and satisfies
\begin{equation}\nonumber
E_{V}(u_0)M(u_0)^{\delta_c}<E_{0}(\mathcal{Q})M(\mathcal{Q})^{\delta_c},\quad \|\nabla u_0\|_{2}\|u_0\|_2^{\delta_c}<\|\nabla \mathcal{Q}\|_{2}\|\mathcal{Q}\|_2^{\delta_c},
\end{equation}
where $\delta_c:=\frac{1-\beta_c}{\beta_c}$ with $\beta_c:=\frac{3}{2}-\frac{2}{p}$.
They also proved that if $2V+x\cdot\nabla V\geq0$,
\begin{equation}\nonumber
E_{V}(u_0)M(u_0)^{\delta_c}<E_{0}(\mathcal{Q})M(\mathcal{Q})^{\delta_c},\quad \|\nabla u_0\|_{2}\|u_0\|_2^{\delta_c}>\|\nabla \mathcal{Q}\|_{2}\|\mathcal{Q}\|_2^{\delta_c},
\end{equation}
and, in addition, either $|x|u_0\in L^2$ or $u_0$ is radial, then the solution blows up in finite time. Recently, the global existence and blow-up criteria for INLS$_{V}$ in inhomogeneous case $-1<c<0$ were also studied by Guo-Wang-Yao \cite{GWY} (case $p=2$) and Dinh \cite{VDD2} (case $\frac{4+2c}{3}<p<4+2c$).
More results on other type of potential have also been addressed, see e.g. Killip-Murphy-Visan-Zheng \cite{KMVZ}, Dinh-Keraani \cite{DK}, An-Kim-Jang \cite{AKJ} for the inverse-power potential $V(x)=\frac{d}{|x|^\sigma}$, or Gustafson-Inui \cite{GI}, Ikeda-Inui \cite{II} for repulsive delta potential.

As we see, the global existence and blow-up for the INLS or INLS$_{V}$ equations have been studied extensively, of which the dispersion coefficient are both $b=0$. While when $b\neq0$, the more general variable coefficient Schr\"{o}dinger equation \eqref{a0} or even for $V=0$ is less understood. Such problem become more subtle in the presence of the dispersion coefficient $|x|^b$.
Motivated by the aforementioned works, we are mainly interested in studying the global existence and blow-up behavior for the focusing INLS$_{b,V}$ in case with $2-n<b<2$ and $2(2-b)\leq\textbf{p}_c<(2-b)(p+2)$.

In this paper, we pose some additional assumptions on the potential $V(x)$ as follows:
\begin{eqnarray}\nonumber
&&(\mbox{\uppercase\expandafter{\romannumeral1}})\quad V(x)\geq0,\quad x\cdot\nabla V+(2-b)V\geq0;\\\nonumber
&&(\mbox{\uppercase\expandafter{\romannumeral2}})\quad x\cdot\nabla V\in L^{\frac{n}{2}}(\Bbb{R}^n; |x|^{-\frac{nb}{2}}dx);\\\nonumber
&&(\mbox{\uppercase\expandafter{\romannumeral3}})\quad x\cdot\nabla V\leq0;\\\nonumber
&&(\mbox{\uppercase\expandafter{\romannumeral4}})\quad x\cdot\nabla^2 Vx^T\leq-(3-b)x\cdot\nabla V;\quad \quad \quad \quad \quad \quad \quad
\end{eqnarray}
where $\nabla^2 V$ is the Hessian matrix of $V$.

Our first goal here is to establish a sufficient condition for the global existence and blow-up of $W_b^{1,2}$-solution to the INLS$_{b,V}$ for radial or finite variance initial data. 
\begin{thm}\label{thmm2}
(Mass-critical case) Let $n\geq3,\ 2-n<b\leq0$, $c\geq b-2,\ \emph{\textbf{p}}_c=2(2-b)$ and let $V(x)$ satisfy (\mbox{\uppercase\expandafter{\romannumeral1}}). Suppose $u$ is the solution to the INLS$_{b,V}$ with initial data $u_0\in W_b^{1,2}$.

(i) If $\|u_0\|_2<\|Q_{1}\|_2$ and $u_0$ is radial, where $Q_{1}$ is the ground state of the following elliptic equation
\begin{equation}\label{c1}
\nabla\cdot(|x|^b\nabla Q_{\omega})-\omega Q_{\omega}+|x|^c|Q_{\omega}|^pQ_{\omega}=0
\end{equation}
with $\omega=1$, then $u$ exists globally and $\sup_{t\in[0, +\infty)}\|u\|_{W_b^{1,2}}<\infty$;

(ii) If $E_{b,V}(u_0)<0$ and $u_0\in L^2(|x|^{2-b}dx)$, then the solution $u$ blows up in finite time.
\end{thm}

\begin{thm}\label{cor1}
(Intercritical case) Let $n\geq3,\ 2-n<b\leq0$, $c\geq b-2,$ $2(2-b)< \emph{\textbf{p}}_c\leq(2-b)(p+2)$ and let $V(x)$ satisfy (\mbox{\uppercase\expandafter{\romannumeral1}})-(\mbox{\uppercase\expandafter{\romannumeral2}}). Let $u(t)$ be the solution to  the INLS$_{b,V}$ with initial data $u_0\in W_b^{1,2}$ and $E_{b,V}(u_0)M(u_0)^\sigma<E_{b}(Q_{1})M(Q_{1})^\sigma$.

(i) If $\|u_0\|_{\dot{H}_{b, V}^1}\|u_0\|_2^\sigma<\|\nabla Q_{1}\|_{b,2}\|Q_{1}\|_2^\sigma$ and $u_0$ is radial, then $u$ exists globally and
\begin{equation}\nonumber
\|u\|_{\dot{H}_{b, V}^1}\|u\|_2^\sigma<\|\nabla Q_{1}\|_{b,2}\|Q_{1}\|_2^\sigma \quad\quad\mbox{for\ any}\ t\in [0, +\infty);
\end{equation}

(ii) If $\|u_0\|_{\dot{H}_{b, V}^1}\|u_0\|_2^\sigma>\|\nabla Q_{1}\|_{b,2}\|Q_{1}\|_2^\sigma$, and either $u_0\in L^2(|x|^{2-b}dx)$ or $u_0$ is radial (in this case, we further assume that $p<4$), then the solution $u(t)$ blows up in finite time and
\begin{equation}\label{c13}
\|u\|_{\dot{H}_{b, V}^1}\|u\|_2^\sigma>\|\nabla Q_{1}\|_{b,2}\|Q_{1}\|_2^\sigma \quad\quad\mbox{for\ any}\ t\in [0, T^\ast).
\end{equation}
\end{thm}
\begin{rem}
 When $u_0$ is radial, the restriction $p<4$ is added to control the nonlinear term $\mathcal{R}_3$ in the localized virial estimate (see Lemma \ref{lem2}).\
\end{rem}

The main tools for proving Theorems \ref{thmm2} and \ref{cor1} are the coercivity given by the virial identities related to the following virial functional
\begin{equation}\label{c4}
P(u):=\|\nabla u\|^2_{b,2}-\frac{1}{2-b}\int (x\cdot\nabla V)|u|^2dx-\frac{\textbf{p}_c}{(2-b)(p+2)}\|u\|^{p+2}_{c,p+2}
\end{equation}
and sharp radial Gagliardo-Nirenberg type estimate.
How to derive the uniform bounds of the functional $P$ in the presence of space-dependent coefficient $|x|^b$ and the term $x\cdot\nabla V$ is the main difficulty we encounter. For this reason, our argument will be more complicated.

In the second part, we aim to extend the global existence and blow-up properties to the INLS$_{b,V}$ with general initial data (not necessarily radial or finite variance). It is conjectured that if a general solution to the INLS$_{b,V}$ satisfies
\begin{equation}\label{}
\sup_{t\in [0, T^\ast)}P(u)\leq-\delta,
\end{equation}
then it blows up. However, there is no affirmative answer on this conjecture up to data for the INLS$_{b,V}$.
To overcome it, we will replace the threshold $E_{b}(Q_{1})M(Q_{1})^\sigma$ in Theorem \ref{cor1} and classify the initial data by using the ground state $Q_{\omega}$ to the elliptic equation \eqref{c1}, which mainly adopts the variational method of Ibrahim-Masmoudi-Nakanishi \cite{IMN}.

To state it, we introduce several notations now. We define the action functional
\begin{eqnarray}\nonumber
\lefteqn{S_{\omega,V}(\phi):=E_{b, V}(\phi)+\frac{\omega}{2}M(\phi)}\\\label{g1}
&&\quad\quad =\frac{1}{2}\|\nabla \phi\|^2_{b,2}+\frac{1}{2}\int V(x)|\phi|^2dx+\frac{\omega}{2}\|\phi\|^2_2-\frac{1}{p+2}\|\phi\|^{p+2}_{c,p+2}.\quad
\end{eqnarray}
By a scaling transformation $\phi\mapsto\phi_\lambda^{\alpha, \beta}$, where
\begin{equation}\label{g35}
\phi_\lambda^{\alpha, \beta}(x):=e^{\alpha\lambda}\phi(e^{\beta\lambda}x),\quad x\in\Bbb{R}^n,
\end{equation}
we define the first order derivative of $S_{\omega,V}(\phi_\lambda^{\alpha, \beta})$ at $\lambda=0$ by $K_{\omega,V}^{\alpha,\beta}(\phi)$
\begin{eqnarray}\nonumber
\lefteqn{K_{\omega,V}^{\alpha,\beta}(\phi):=\frac{\partial}{\partial\lambda}S_{\omega,V}(\phi_\lambda^{\alpha, \beta})\mid_{\lambda=0}}\\\nonumber
&&\quad\quad =\frac{2\alpha+(2-b-n)\beta}{2}\|\nabla\phi\|^2_{b,2}+\frac{2\alpha-n\beta}{2}(\int V(x)|\phi|^2dx+\omega\|\phi\|^2_2)\\\label{g3}
&&\quad\quad\quad\ -\frac{\beta}{2}\int (x\cdot\nabla V)|\phi|^2dx-\frac{\alpha(p+2)-(n+c)\beta}{p+2}\|\phi\|^{p+2}_{c,p+2}.
\end{eqnarray}
In particular, when $(\alpha, \beta)=(n,2)$ and $(\alpha,\beta)=(1,0)$,
\begin{eqnarray}\nonumber
&&P(\phi):=\frac{1}{2-b}K_{\omega,V}^{n,2}(\phi)\\\nonumber
&&\quad\quad\quad=\|\nabla \phi\|^2_{b,2}-\frac{1}{2-b}\int (x\cdot\nabla V)|\phi|^2dx-\frac{\textbf{p}_c}{(2-b)(p+2)}\|\phi\|^{p+2}_{c,p+2},\\\nonumber
&&I_{\omega,V}(\phi):=K_{\omega,V}^{1,0}(\phi)=\|\nabla \phi\|^2_{b,2}+\int V(x)|\phi|^2dx+\omega\|\phi\|^2_2-\|\phi\|^{p+2}_{c,p+2},\quad\quad\quad
\end{eqnarray}
where $P$ is the virial functional defined in \eqref{c4}, and $I_{\omega,V}$ is called Nehari functional.

For each $(\alpha, \beta)\in\Bbb{R}^2$ in the range
\begin{equation}\label{g2}
\alpha>0,\quad \beta\geq0,\quad 2\alpha-n\beta\geq0
\end{equation}
and the elliptic equation \eqref{c1}, we consider the following constrained minimizing problem for nonradial function
\begin{equation}\label{g6}
m_{\omega,V}^{\alpha,\beta}=\inf\{S_{\omega,V}(\phi):\ \phi\in W_b^{1,2}\backslash\{0\},\ K_{\omega,V}^{\alpha,\beta}(\phi)=0\}.
\end{equation}
As in \cite{ZZ1}, the minimizing problem $m_{\omega,0}^{1,0}$ is proved to be attained by the ground state $Q_{\omega}$. However, the more general case $m_{\omega,0}^{\alpha,\beta}$ for $(\alpha, \beta)$ satisfying \eqref{g2} seems to be unknown in the study of stability analysis, which is essential in our proof of global existence and blow-up of the INLS$_{b,V}$.

In order to do so, we first investigate some properties of the ground state $Q_{\omega}$ and make a comparison on the minimization problems between $m_{\omega,0}^{\alpha,\beta}$ and $m_{\omega,V}^{\alpha,\beta}$.
\begin{pro}\label{prop1}
Let $n\geq3,\ 2-n<b<2,\ b-2<c\leq\min\{\frac{nb}{n-2},\ 0\},\ \omega>0,\ 2(2-b)<\emph{\textbf{p}}_c<(2-b)(p+2)$ and let $(\alpha, \beta)$ satisfy \eqref{g2}.

(i)  If $V=0$, then $m_{\omega,0}^{\alpha,\beta}$ is attained by the ground state $Q_{\omega}$ of the elliptic equation \eqref{c1};

(ii) If $V(x)$ satisfies (\uppercase\expandafter{\romannumeral1}) and (\uppercase\expandafter{\romannumeral3}), then $m_{\omega,0}^{\alpha,\beta}\leq m_{\omega,V}^{\alpha,\beta}$.
\end{pro}
Proposition \ref{prop1} shows that $m_{\omega,0}^{\alpha,\beta}$ is independent of $(\alpha, \beta)$. From now on, we express $m_{\omega,0}:=m_{\omega,0}^{\alpha,\beta}$ for simplicity.

Now by defining two subsets in $W_b^{1,2}$ as follows:
\begin{equation}\label{g23}
\mathcal{N}^+=\{\phi\in W_b^{1,2}: \ S_{\omega,V}(\phi)<m_{\omega,0},\quad K_{\omega,V}^{n,2}(\phi)\geq0\}
\end{equation}
and
\begin{equation}\label{g24}
\mathcal{N}^-=\{\phi\in W_b^{1,2}: \ S_{\omega,V}(\phi)<m_{\omega,0},\quad K_{\omega,V}^{n,2}(\phi)<0\},
\end{equation}
we prove the global existence and blow-up of the nonradial solution to the INLS$_{b,V}$, whose action is less than $m_{\omega,0}$.
\begin{thm}\label{thmm1}
Let $n\geq3,\ 2-n<b<2,\ b-2<c\leq0,\ \omega>0,\ 2(2-b)<\emph{\textbf{p}}_c<(2-b)(p+2)$ and let $V(x)$ satisfy (\uppercase\expandafter{\romannumeral1})-(\uppercase\expandafter{\romannumeral4}).
Let $u\in C([0, T^\ast); W_b^{1,2})$ be the local solution to the INLS$_{b,V}$ with initial data $u_0\in W_b^{1,2}$.

(i) If $u_0\in \mathcal{N}^+$, then the solution exists globally and $u(t)\in \mathcal{N}^+$ for any $t\in[0, +\infty)$.

(ii) If $u_0\in \mathcal{N}^-$ and additionally
\[\quad\quad c\leq b\leq0,\quad\quad \emph{\textbf{p}}_c\leq\frac{2c(2-b)}{b},\]
then $u(t)\in \mathcal{N}^-$ for any $t\in[0, T^\ast)$ and one of the following finite/infinite time blow-up statement holds true:

\quad (1) $T^\ast<+\infty$ and $\lim_{t\rightarrow T^\ast}\|\nabla u(t)\|_{b,2}=+\infty$;

\quad (2) $T^\ast=+\infty$ and there exists a time sequence $\{t_n\}$ such that $t_n\rightarrow+\infty$ and
\[\lim_{n\rightarrow +\infty}\|\nabla u(t_n)\|_{b,2}=+\infty.\]
\end{thm}
\begin{rem}
There is a restriction $\emph{\textbf{p}}_c\leq\frac{2c(2-b)}{b}$ for the blow-up. This restriction is due to the Hardy-Sobolev inequality in \eqref{c26}.
\end{rem}

As a byproduct of the above criteria, we unify the condition $u_0\in\mathcal{N}^-$ with \eqref{g40}, and obtain the following nonradial blow-up result for the INLS$_{b,V}$, which is a version independent of the frequency.
\begin{corollary}\label{thm4}
Let $n,\ b,\ c$, $\emph{\textbf{p}}_c$ and $V(x)$ be the same as in Theorem \ref{thmm1} (ii).
Let $u$ be the corresponding solution to the INLS$_{b,V}$ with initial data $u_0\in W_b^{1,2}$ and satisfying
\begin{equation}
\left\{
\begin{aligned}
 & E_{b,V}(u_0)M(u_0)^\sigma<E_{b}(Q_{1})M(Q_{1})^\sigma,\\\label{g40}
 &\|u_0\|_{\dot{H}_{b, V}^1}\|u_0\|_2^\sigma>\|\nabla Q_{1}\|_{b,2}\|Q_{1}\|_2^\sigma.
\end{aligned}\right.
\end{equation}
Then either $T^\ast<+\infty$, or $T^\ast=+\infty$ and there exists a time sequence $t_n\rightarrow+\infty$ such that $\|\nabla u(t_n)\|_{b,2}\rightarrow+\infty$ as $n\rightarrow +\infty$.
\end{corollary}

This paper is organized as follows. In Section 2, we are devoted to preparatory materials, including the weighted Sobolev embedding and virial estimates adapted to the INLS$_{b,V}$.
In Section 3, we prove the global existence and blow-up for the INLS$_{b,V}$ with radial initial data given in Theorems \ref{thmm2} and \ref{cor1}. In Section 4, we show the variational analysis given in Proposition \ref{prop1}. Finally, the global existence and blow-up for nonradial solution are given in Section 5.

\textbf{Notations.} We use $L^q(\Bbb{R}^n; \omega(x)dx)$ to denote the weighted Lebesgue space, which is defined via
\[\|u\|_{L^q(\Bbb{R}^n;\ \omega(x)dx)}=(\int\omega(x)|u|^qdx)^{\frac{1}{q}}.\]
To shorten formulas, we often abbreviate $L^q(\Bbb{R}^n;\ \omega(x)dx)$ by $L^q(\omega(x)dx)$, the norm $\|\cdot\|_{L^q(\Bbb{R}^n; |x|^adx)}$ by $\|\cdot\|_{a,q}$. In particular, we denote $\|\cdot\|_{q}=\|\cdot\|_{0,q}$.

We also define by the weighted Sobolev space $W_a^{s,q}(\Bbb{R}^n)$ by the completion of $C_0^\infty(\Bbb{R}^n)$ with respect to the norm
\[\|u\|_{W_{a}^{s,q}}=\|(-\Delta)^{\frac{s}{2}} u\|_{a,q}+\|u\|_{q}\quad\mbox{and}\quad  \|u\|_{\dot{W}_{a}^{s,q}}=\|(-\Delta)^{\frac{s}{2}} u\|_{a,q}.\]
We make the usual modifications when $q$ equals to $2$.
In particular, $H^s=W_0^{s,2},\ \dot{H}^s=\dot{W}_{0}^{s,2},$ where $H^s$ (\emph{or $\dot{H}^s$}) are usual inhomogeneous (\emph{or homogeneous}) Sobolev spaces. We denote $\delta_{ij}=1$ if $i=j$ and $0$ if $i\neq j$. If not specified, we use $C_\alpha$ to denote various constants which depend on $\alpha$.

\section{Preliminary}
\subsection{Key tools}
The following sharp Gagliardo-Nirenberg inequality for radial functions $f\in W^{1,2}_{b}$ can be found in \cite{ZOZ}.
\begin{lemma}
Let $2-n<b\leq0$ and let one of the following conditions
\begin{eqnarray}\label{g42}
&&(1)\ \ b-2\leq c\leq0:\quad\  -2c<\emph{\textbf{p}}_c<(2-b)(p+2),\\\label{b6}
&&(2)\ \ c>0:\quad\quad\quad\quad\ \frac{(2-b)p}{2}<\emph{\textbf{p}}_c<(2-b)(p+2)\quad\quad\quad\quad\quad\quad
\end{eqnarray}
holds. Then for any radial function $f\in W^{1,2}_{b}$, we have
\begin{equation}\label{a19}
\|f\|_{c,p+2}^{p+2}\leq C_{GN}\|\nabla f\|_{b,2}^{\frac{\emph{\textbf{p}}_c}{2-b}}\|f\|_{2}^{\frac{(2-b)(p+2)-\emph{\textbf{p}}_c}{2-b}},
\end{equation}
where $C_{GN}=\left(\frac{\omega\emph{\textbf{p}}_c}{(2-b)(p+2)-\emph{\textbf{p}}_c}\right)^{1-\frac{\emph{\textbf{p}}_c}{2(2-b)}}\frac{(2-b)(p+2)}{\emph{\textbf{p}}_c\|Q_\omega\|_2^p}$.
The equality is attained by $Q_{\omega}$, which is the ground state solution to the elliptic equation \eqref{c1}.

\end{lemma}
\begin{rem}
(1) We also have the following Pohozaev's identities:
\begin{eqnarray}\label{a3}
\|Q_{\omega}\|_2^2=\frac{(2-b)(p+2)-\emph{\textbf{p}}_c}{(2-b)(p+2)\omega}\|Q_{\omega}\|_{c,p+2}^{p+2}=\frac{(2-b)(p+2)-\emph{\textbf{p}}_c}{\omega\emph{\textbf{p}}_c}\|\nabla Q_{\omega}\|^2_{b,2}.
\end{eqnarray}
In particular, we have
\begin{eqnarray}\nonumber
C_{GN}=\frac{(2-b)(p+2)}{\emph{\textbf{p}}_c}\left(\|\nabla Q_{\omega}\|_{b,2}\|Q_{\omega}\|_{2}^\sigma\right)^{2-\frac{\emph{\textbf{p}}_c}{2-b}}.
\end{eqnarray}

(2) We note that the Gagliardo-Nirenberg inequality for nonradial functions $f\in W^{1,2}_{b}$
\begin{equation}\label{g7}
\|f\|_{c,p+2}^{p+2}\leq C\|\nabla f\|_{b,2}^{\frac{\textbf{p}_c}{2-b}}\|f\|_{2}^{\frac{(2-b)(p+2)-\emph{\textbf{p}}_c}{2-b}}
\end{equation}
and $-n<c\leq0,\ \emph{\textbf{p}}_c\leq(2-b)(p+2)$ can be derived from the celebrated Caffarelli-Kohn-Nirenberg inequality (see \cite{CKN}), which also occurred in \cite{ZO}.
\end{rem}

In the case where $\Omega\subset\Bbb{R}^n$ is bounded, the weighted Sobolev embedding lemma has been proved in \cite{GGW}.
\begin{lemma}\label{lem3}
Let $n\geq3$ and let $\Omega\subset\Bbb{R}^n$ be an open bounded set. Assume that $b>2-n$ and $b-2<c\leq\frac{nb}{n-2}$. Then the embedding
\[\dot{W}_b^{1,2}(\Omega)\hookrightarrow L^q(\Omega;\ |x|^cdx)\]
is continuous provided $q\in[1, \frac{2n+2c}{n-2+b}]$, and this embedding is compact for $q\in[1, \frac{2n+2c}{n-2+b})$.
\end{lemma}

Now we extend the above compact embedding lemma to the whole Euclidean space, which mainly adopts the idea in \cite{ZO}. We include the proof for completeness.
\begin{lemma}\label{lemma6}
Let $n\geq3$, $2-n<b<2,\ b-2<c\leq \frac{nb}{n-2}$. If the function $\omega$ satisfies
\begin{equation}\label{f68}
\limsup_{|x|\rightarrow0}\frac{\omega(x)}{|x|^c}<\infty\quad\quad\mbox{and}\quad\quad \limsup_{|x|\rightarrow\infty}\frac{\omega(x)}{|x|^c}<\infty,
\end{equation}
then the following embedding
\begin{equation}\label{d1}
W_b^{1,2}\hookrightarrow L^{p+2}(\omega(x)dx)
\end{equation}
is compact for all $2(2-b)<\emph{\textbf{p}}_c<(2-b)(p+2)$.
\end{lemma}
\begin{proof}
Let $\{f_n\}_{n\in\Bbb{N}}$ be a bounded sequence in $W_b^{1,2}$. Then there exists a function $f\in W_b^{1,2}$ such that, up to subsequence,
\begin{equation}\label{f69}
f_n\rightharpoonup f\quad\quad \mbox{in}\ \ W_b^{1,2}\quad \mbox{as}\ n\rightarrow\infty.
\end{equation}
Define $g_n=f_n-f$. It suffices to show that
\[\lim_{n\rightarrow\infty}\int \omega(x)|g_n|^{p+2}dx=0.\]

To do it, by \eqref{f68}, there exists $R_1>r_1>0$ and various constants $C>0$ such that
\begin{equation}\nonumber
\omega(x)\leq C|x|^c\quad\quad\ \mbox{for any}\ x\ \mbox{with}\ 0<|x|\leq r_1\quad
\end{equation}
and
\begin{equation}\nonumber
\omega(x)\leq C|x|^c\quad\quad\ \ \mbox{for any}\ x\ \mbox{with}\ |x|\geq R_1.\quad\quad
\end{equation}
Now we choose a smooth cut-off function $\psi\in C_0^\infty(\Bbb{R}^n)$ satisfying
\begin{equation}\psi(x)=\nonumber
\left\{
\begin{aligned}
 &1, \quad\ |x|\leq\frac{r_1}{2},\\
 & 0, \quad\ |x|\geq r_1,
\end{aligned}\right.\ \quad\mbox{and}\quad\quad |\nabla\psi|\leq C.
\end{equation}
Then for $0<r<\frac{r_1}{2}$, we get
\begin{eqnarray}\nonumber
\lefteqn{\int_{|x|\leq r}\omega(x)|g_n|^{p+2}dx\leq Cr^{\frac{(2-b)(p+2)-\textbf{p}_c}{2}}\int|x|^{\frac{np-(2-b)(p+2)}{2}}|\psi g_n|^{p+2}dx}\\\nonumber
&&\quad\quad\quad\quad\quad\quad\quad \leq Cr^{\frac{(2-b)(p+2)-\textbf{p}_c}{2}}(\int|x|^{b}|\nabla(\psi g_n)|^2dx)^{\frac{p+2}{2}}\quad\quad
\end{eqnarray}
for all $2(2-b)<\textbf{p}_c<(2-b)(p+2)$ and $b>2-n$, where we used the Hardy-Sobolev inequality \eqref{c26} at the last step.
So, given any small $\epsilon>0$ and combined with  \eqref{f69}, there exists $r_1>r>0$ such that
\begin{equation}\label{d10}
\int_{|x|\leq r}\omega(x)|g_n|^{p+2}dx<\frac{\epsilon}{3}.
\end{equation}

Next, we will control the integral $\int\omega(x)|g_n|^{p+2}dx$ out of a big ball. For $R>R_1$, we have
\begin{equation}\label{d6}
\int_{|x|>R}\omega(x)|g_n|^{p+2}dx\leq\frac{C}{R^{\frac{2-b}{2}}}\int_{|x|>R}|x|^{c+\frac{2-b}{2}}|g_n|^{p+2}dx.
\end{equation}

To estimate \eqref{d6}, we invoke the Caffarelli-Kohn-Nirenberg inequality to obtain
\begin{equation}\nonumber
\int|x|^{c+\frac{2-b}{2}} |g_n|^{p+2}dx\leq C\|\nabla g_n\|_{b,2}^{a(p+2)}\|g_n\|_{2}^{(1-a)(p+2)}
\end{equation}
for $p\in(\frac{2c+2-b}{n},\ \frac{2c+3(2-b)}{n-2+b})$ and $b>2-n$, where $a=\frac{\textbf{p}_c-2+b}{(2-b)(p+2)}\in(0,1).$
Thus we have proved that
\begin{equation}\label{d7}
\{g_n\}_{n\in\Bbb{N}}\ \mbox{is\ uniformly\ bounded\ in}\ L^{p+2}(|x|^{c+\frac{2-b}{2}}dx)
\end{equation}
for all $p\in(\frac{2(2-b+c)}{n},\ \frac{2(2-b+c)}{n-2+b})$. 
Therefore, for any $\epsilon>0$, it follows from \eqref{d6} and \eqref{d7} that there exists $R>R_1$ such that
\begin{equation}\label{g39}
\int_{|x|>R}\omega(x)|g_n|^{p+2}dx<\frac{\epsilon}{3}.
\end{equation}

Finally, for $R>r>0$ given above, we deduce from Lemma \ref{lem3} that the embedding
\[\dot{W}_b^{1,2}(r<|x|\leq R)\hookrightarrow L^{p+2}(r<|x|\leq R; |x|^cdx)\]
is compact for $p<\frac{2(2-b+c)}{n-2+b}$. This together with \eqref{f69} yields that
\begin{equation}\label{f70}
\int_{r<|x|\leq R}\omega(x)|g_n|^{p+2}dx<\frac{\epsilon}{3}.
\end{equation}
Therefore, putting \eqref{d10}, \eqref{g39} and \eqref{f70} all together, we conclude that the embedding \eqref{d1} is compact.
\end{proof}

As a consequence, for $\omega(x)=|x|^c$ with $b-2<c\leq\frac{nb}{n-2}$, Lemma \ref{lemma6} implies that the embedding
\begin{equation}
W_b^{1,2}\hookrightarrow L^{p+2}(|x|^cdx)
\end{equation}
is compact for $2(2-b)<\textbf{p}_c<(2-b)(p+2)$.

At the end, we recall the radially weighted Strauss inequality cited in \cite{ZOZ}.
\begin{lemma}\label{lem3.1}
Let $b\geq2-2n$ and $f\in W_b^{1,2}$ be a radial function. Then we have the following inequality
\begin{equation}\nonumber
\sup_{x\in \Bbb{R}^n}|x|^{\frac{2n-2+b}{4}}|f|\leq C_n\|\nabla f\|_{b,2}^{\frac{1}{2}}\|f\|_{2}^{\frac{1}{2}}.
\end{equation}
\end{lemma}

The following known Hardy-Sobolev inequality appears in \cite{WW}.
\begin{lemma}\label{lem11}
Let $n\geq3,\ b>2-n$ and $-\frac{b}{2}\leq d\leq \frac{2-b}{2}$. Then
\begin{equation}\label{c26}
(\int |x|^{-\frac{2nd}{n-2+b+2d}}|f|^{\frac{2n}{n-2+b+2d}}dx)^{\frac{n-2+b+2d}{n}}\leq C_{b,d}\int |x|^b|\nabla f|^2dx.
\end{equation}

\end{lemma}

\subsection{Virial identities}
Given a real value function $\psi$, we define
\begin{equation}\label{e5}
I_\psi(t)=\int\psi(x)|u|^2dx.
\end{equation}
A direct computation shows that 
\begin{lemma}\label{lem2.1}
Let $u\in C(I; W_b^{1,2})$ be the solution to the INLS$_{b,V}$. Assume that $\psi,\ \phi\in W^{4,\infty}$ and  satisfy
\[\nabla\psi(x)=\frac{\nabla\phi(x)}{|x|^b}.\] Then
\begin{equation}\label{b26}
I'_\psi(t)=2\emph{Im}\int\nabla\phi\cdot\nabla u\overline{u}dx
\end{equation}
and
\begin{eqnarray}\nonumber
&&I''_\psi(t)=-\frac{2}{p+2}\int\left(p|x|^c\Delta\phi-2\nabla\phi\cdot\nabla|x|^c\right)|u|^{p+2}dx\\\nonumber
&&\quad\quad\quad\ -2\int(\nabla\phi\cdot\nabla |x|^b)|\nabla u|^2dx-2\int\nabla\phi\cdot\nabla V|u|^2dx\\\nonumber
&&\quad\quad\quad\ +4\emph{Re}\int(\nabla^2\phi\cdot\nabla u)\cdot\nabla \overline{u}|x|^bdx\\\label{b11}
&&\quad\quad\quad\ -\int(|x|^b\Delta^2\phi+\nabla\Delta\phi\cdot\nabla|x|^b)|u|^2dx.\quad\quad\quad\
\end{eqnarray}
\end{lemma}
\begin{proof}
The proof is standard, see, for instance \cite{CH}, which deals with the variable coefficient Schr\"{o}dinger equation with harmonic potential $V(x)=|x|^2$. Here we omit the details.
\end{proof}

As a consequence, we obtain the following virial identity for the INLS$_{b,V}$.
\begin{corollary}\label{cor2}
Let $u_0\in W_b^{1,2}\cap L^2(|x|^{2-b}dx)$. Assume $u\in C(I; W_b^{1,2})$ is the maximal lifespan solution to the INLS$_{b,V}$. Then $u\in C(I; L^2(|x|^{2-b}dx))$ and for any $t\in I$,
\begin{equation}\nonumber
\frac{d^2}{dt^2}\|u(t)\|^2_{2-b,2}=2(2-b)^2\|\nabla u\|^2_{b,2}-2(2-b)\int (x\cdot\nabla V)|u|^2dx-\frac{2(2-b)\emph{\textbf{p}}_c}{p+2}\|u\|^{p+2}_{c,p+2}.
\end{equation}
\end{corollary}

\begin{proof}
Let $\varepsilon>0,$ we multiply the INLS$_{b,V}$ equation \eqref{a0} by $ie^{-2\varepsilon |x|^{2-b}}|x|^{2-b}\overline{u}$, integrate over $\Bbb{R}^n$ and take the real part
of the result to obtain
\begin{equation}\nonumber
\frac{d}{dt}\int e^{-2\varepsilon |x|^{2-b}}|x|^{2-b}|u|^2dx=2\textmd{Im}\int|x|^{b}\nabla u\cdot\nabla(e^{-2\varepsilon |x|^{2-b}}|x|^{2-b}\overline{u})dx.
\end{equation}
Integrating the above equality with respect to $t$, we have
\begin{eqnarray}\label{e12}
\lefteqn{\int e^{-2\varepsilon |x|^{2-b}}|x|^{2-b}|u|^2dx=\int e^{-2\varepsilon |x|^{2-b}}|x|^{2-b}|u_0|^2dx}\\\nonumber
&&\quad\quad\quad\quad\quad\quad\quad\quad\quad+2(2-b)\textmd{Im}\int_0^t\int e^{-2\varepsilon |x|^{2-b}}\overline{u}x\cdot\nabla u(1-2\varepsilon |x|^{2-b})dxdt.
\end{eqnarray}
Since $ e^{-\varepsilon |x|^{2-b}}\left(1-2\varepsilon |x|^{2-b}\right)$ is bounded in both $x$ and $\varepsilon$, and that $\nabla u$ is bounded in $L^2(|x|^bdx)$, so it follows from \eqref{e12} that for every compact interval $I\subset[0, T^\ast)$, there exists constant $C$ such that
\[\int e^{-2\varepsilon |x|^{2-b}}|x|^{2-b}|u|^2dx\leq C\int e^{-2\varepsilon |x|^{2-b}}|x|^{2-b}|u_0|^2dx<\infty\]
for all $\varepsilon>0$. It follows that the function $t\mapsto u(t)$ is weakly continuous in $[0, T^\ast)\rightarrow L^2(|x|^{2-b}dx)$. Now let $\varepsilon\rightarrow0$ in \eqref{e12}, then
\begin{equation}\nonumber
\int|x|^{2-b}|u|^2dx=\int|x|^{2-b}|u_0|^2dx+2(2-b)\textmd{Im}\int_0^t\int\overline{u}x\cdot\nabla udxdt,
\end{equation}
which yields that the function $t\mapsto u(t)$ is continuous in $[0, T^\ast)\rightarrow L^2(|x|^{2-b}dx)$.
Hence, we conclude that the solution $u\in C([0, T^\ast);\ L^2(|x|^{2-b}dx))$.

It remains to show the virial identity. Let $\nabla\phi=x$ in Lemma \ref{lem2.1}, then
\begin{eqnarray}\nonumber
\lefteqn{\frac{d^2}{dt^2}\|u(t)\|^2_{2-b,2}=2(2-b)^2\|\nabla u\|^2_{b,2}}\\\nonumber
&&\quad\quad\quad\quad\quad\quad\quad-2(2-b)\int (x\cdot\nabla V)|u|^2dx-\frac{2(2-b)\textbf{p}_c}{p+2}\|u\|^{p+2}_{c,p+2},
\end{eqnarray}
which is the desired result.
\end{proof}

Next, we establish a localized virial estimate related to the INLS$_{b,V}$ equation \eqref{a0}.

If we suppose $\psi$ and $\phi$ are radially symmetric, then Lemma \ref{lem2.1} implies
\begin{equation}\label{c2}
I'_\psi(t)=2\textmd{Im}\int\phi'\frac{x\cdot\nabla u}{r}\overline{u}dx,
\end{equation}
where $r=|x|$. Moreover, since
\begin{eqnarray}\nonumber
\lefteqn{4\textmd{Re}\int(\nabla^2\phi\cdot\nabla u)\cdot\nabla \overline{u}|x|^b dx=4\int\frac{\phi'}{r}|x|^b|\nabla u|^2dx}\\\nonumber
&&\quad\quad\quad\quad\quad\quad\quad\quad\quad\quad\quad\quad+4\int(\frac{\phi''}{r^2}-\frac{\phi'}{r^3})|x|^b|x\cdot\nabla u|^2dx,\quad\quad
\end{eqnarray}
which is obtained by the fact that
\[\frac{\partial}{\partial x_j}=\frac{x_j}{r}\partial_r,\quad\quad\quad \frac{\partial^2}{\partial x_j\partial x_k}=(\frac{\delta_{jk}}{r}-\frac{x_jx_k}{r^3})\partial_r+\frac{x_jx_k}{r^2}\partial_r^2\]
on radial functions. Thus, \eqref{b11} can be rewritten as
\begin{eqnarray}\nonumber
\lefteqn{I''_\psi(t)=4\int(\frac{\phi''}{r^2}-\frac{\phi'}{r^3})|x|^b|x\cdot\nabla u|^2dx}\\\nonumber
&&\quad+(4-2b)\int\frac{\phi'}{r}|x|^b|\nabla u|^2dx-\frac{2}{p+2}\int(p\Delta\phi-\frac{2c\phi'}{r})|x|^c|u|^{p+2}dx\\\label{c3}
&&\quad-2\int\frac{\phi'}{r}(x\cdot\nabla V)|u|^2dx-\int(|x|^b\Delta^2\phi+\nabla\Delta\phi\cdot\nabla|x|^b)|u|^2dx.
\end{eqnarray}

By choosing an appropriate cut-off function similar to $x^2$ near the origin $x=0$, we deduce the following localized virial estimate.
\begin{lemma}\label{lem2}
Let $n\geq3$ and $b>2-n$. If $x\cdot\nabla V\in L^{\frac{n}{2}}(|x|^{-\frac{nb}{2}}dx)$ and $u\in C([0, T^\ast); W_b^{1,2})$ is the solution to the INLS$_{b,V}$ with radial initial data $u_0$, then there exists constant $C>0$ depending on $\emph{\textbf{p}}_c, b, M(u_0)$ such that the following localized virial estimate
\begin{eqnarray}\nonumber
\lefteqn{I''_{\psi_R}(t)\leq4(2-b)P(u)+C_{\epsilon}R^{-\frac{2\textbf{p}_c-(2-b)p}{4-p}}M(u_0)^{\frac{p}{4}+1}}\\\label{e23}
&&\ +CR^{b-2}M(u_0)+(2b+C\epsilon+C\|x\cdot\nabla V\|_{L^{\frac{n}{2}}(|x|>R;\ |x|^{-\frac{nb}{2}}dx)})\|\nabla u\|_{b,2}^2\quad\quad\
\end{eqnarray}holds for any $t\in[0, T^\ast)$.
\end{lemma}
\begin{proof}
We consider the radial function $\phi_R$ constructed by $\phi_R(x)=R^2\Theta(\frac{x}{R})$, where
\begin{equation}\Theta(x)=
\left\{
\begin{aligned}
 & |x|^2,\quad\quad\quad \mbox{if}\quad |x|\leq 1,\\\label{d13}
 & \mbox{smooth},\quad\ \mbox{if}\quad 1<|x|<2,\\
 & 0,\quad\quad\quad\quad \mbox{if}\quad |x|\geq2
\end{aligned}\right.
\end{equation}
and $\|\Theta\|_{W^{2,\infty}}\leq2$. Applying $\nabla\psi_R=\frac{\nabla\phi_R}{|x|^b}$ to \eqref{c3}, we can rewrite $I''_{\psi_R}(t)$ as
\begin{equation} \label{e13}
I''_{\psi_R}(t)=4(2-b)P(u)+\mathcal{R}_1+\mathcal{R}_2+\mathcal{R}_3+\mathcal{R}_4,
\end{equation}
where $\mathcal{R}_j\ (j=1,2,3,4)$ are defined by
\begin{eqnarray}\nonumber
&&\mathcal{R}_1=4\int_{|x|> R}(\frac{\phi''_R}{r^2}-\frac{\phi'_R}{r^3})|x|^b|x\cdot\nabla u|^2dx+(4-2b)\int_{|x|> R}(\frac{\phi'_R}{r}-2)|x|^b|\nabla u|^2dx,\\\nonumber
&&\mathcal{R}_2=-\int_{|x|> R}(|x|^b\Delta^2\phi_R+\nabla\Delta\phi_R\cdot\nabla|x|^b)|u|^2 dx,
\end{eqnarray}
\begin{eqnarray}\nonumber
&&\mathcal{R}_3=-\frac{2}{p+2}\int_{|x|>R}(p\Delta\phi_R-\frac{2c\phi_R'}{r})|x|^c|u|^{p+2}dx+\frac{4\textbf{p}_c}{p+2}\int_{|x|> R}|x|^c|u|^{p+2}dx,\\\nonumber
&&\mathcal{R}_4=2\int_{|x|> R}(2-\frac{\phi_R'}{r})(x\cdot\nabla V)|u|^2dx.\quad\quad\quad\quad\quad\quad\quad\quad\
\end{eqnarray}
Similarly to the argument of Lemma 2.8 in \cite{ZZ1}, which deals with the INLS$_{b,V}$ without potential, we easily get
\begin{eqnarray}\label{c5}
&&\mathcal{R}_1\leq2b\int_{|x|> R}(2-\frac{\phi'_R}{r})|x|^b|\nabla u|^2dx,\\\label{d11}
&&\mathcal{R}_2\leq C_{n,b}R^{b-2}\|u\|_2^2
\end{eqnarray}
and
\begin{equation}\nonumber
\mathcal{R}_3\leq C_{\textbf{p}_c} \int_{|x|> R}|x|^c|u|^{p+2}dx\leq C_{\textbf{p}_c}R^{c-\frac{(2n-2+b)p}{4}}\|\nabla u\|_{b,2}^{\frac{p}{2}}\|u\|^{\frac{p}{2}+2}_2,
\end{equation}
where the last inequality in $\mathcal{R}_3$ is obtained by Lemma \ref{lem3.1}.
Thus, applying the Young inequality to $\mathcal{R}_3$, we further have
\begin{equation}\label{e18}
\mathcal{R}_3\leq C_{\epsilon,\textbf{p}_c}R^{-\frac{2\textbf{p}_c-(2-b)p}{4-p}}\|u\|^{\frac{p}{2}+2}_2+C\epsilon\|\nabla u\|_{b,2}^2\|u\|^{\frac{p}{2}+2}_2.
\end{equation}

The remaining task is to control the potential part $\mathcal{R}_4$. Since $2-\frac{\phi'_R}{r}\geq0$, we use the H\"{o}lder inequality to get
\begin{eqnarray}\nonumber
\lefteqn{\mathcal{R}_4\ \leq C\int_{|x|> R}|x\cdot\nabla V||u|^2dx}\\\nonumber
&&\leq C(\int_{|x|> R}|x|^{-\frac{nb}{2}}|x\cdot\nabla V|^{\frac{n}{2}}dx)^{\frac{2}{n}}(\int_{|x|> R}|x|^{\frac{nb}{n-2}}|u|^{\frac{2n}{n-2}}dx)^{\frac{n-2}{n}},
\end{eqnarray}
which together with \eqref{c26} yields
\begin{equation}\label{e17}
\mathcal{R}_4\leq C_{b}\|x\cdot\nabla V\|_{L^{\frac{n}{2}}(|x|>R;\ |x|^{-\frac{nb}{2}}dx)}\|\nabla u\|_{b,2}^2.
\end{equation}
Therefore, putting \eqref{c5}-\eqref{e17} all together, we conclude the result of Lemma \ref{lem2}.
\end{proof}

\section{Global existence and blow-up for radial case}
In this section, we are devoted to the proof of Theorems \ref{thmm2} and \ref{cor1}.
\subsection{Global existence}
In this part, we study the global existence for the mass-critical and intercritical INLS$_{b,V}$ with radial initial data, upon the LWP assumption in $W_b^{1,2}$ stated in Section 1.

\textbf{Proof of the global parts in Theorems \ref{thmm2} and \ref{cor1}.}
Let $u\in C([0, T^\ast); W_b^{1,2})$ be the solution to the INLS$_{b,V}$ with $u_0\in  W_b^{1,2}$. In view of the conservation laws, we just need to bound the $W_b^{1,2}$-norm of $u(t)$ for any $t\in [0, T^\ast)$.

\textbf{Case 1} (Mass-critical). Since $V(x)\geq0$, we deduce from the energy conservation that
\begin{eqnarray}\nonumber
\|\nabla u\|^2_{b,2}\leq\|\nabla u\|^2_{b,2}+\int V(x)|u|^2dx=2E_{b,V}(u_0)+\frac{2}{p+2}\|u\|^{p+2}_{c,p+2}.
\end{eqnarray}
Invoking the sharp radial Gagliardo-Nirenberg inequality \eqref{a19}, we further get
\begin{eqnarray}\nonumber
\lefteqn{\|\nabla u\|^2_{b,2}\leq2E_{b,V}(u_0)+\frac{2C_{GN}}{p+2}\|\nabla u\|_{b,2}^{\frac{\textbf{p}_c}{2-b}}\|u\|_{2}^{\frac{(2-b)(p+2)-\textbf{p}_c}{2-b}}}\\\nonumber
&&\quad\quad=2E_{b,V}(u_0)+\left(\frac{\|u_0\|_2}{\|Q_{\omega}\|_2}\right)^p\|\nabla u\|_{b,2}^2.
\end{eqnarray}
Since $\|u_0\|_2<\|Q_{\omega}\|_2$, we obtain that $\|\nabla u\|_{b,2}$ is bounded for any $t\in [0, T^\ast)$, which derives the global existence in Theorem \ref{thmm2} (i).

\textbf{Case 2} (Intercritical). Multiplying both sides of $E_{b,V}(u)$ by $M(u)^\sigma$ and using the sharp radial Gagliardo-Nirenberg inequality \eqref{a19},
we have %for $\mu<0$,
\begin{eqnarray}\nonumber
\lefteqn{E_{b,V}(u)M(u)^\sigma=\frac{1}{2}(\|u\|_{\dot{H}_{b,V}^1}\|u\|_{2}^\sigma)^2-\frac{1}{p+2}\|u\|_{c,p+2}^{p+2}\|u\|_2^{2\sigma}}\\\label{c14}
&&\quad\quad\quad\quad\quad\geq\frac{1}{2}\left(\|u\|_{\dot{H}_{b,V}^1}\|u\|_{2}^\sigma\right)^2-\frac{ C_{GN}}{p+2}(\|u\|_{\dot{H}_{b,V}^1}\|u\|_{2}^\sigma)^{\frac{\textbf{p}_c}{2-b}},
\end{eqnarray}
where we used the fact that $\|u\|_{\dot{H}_{b,V}^1}\geq\|\nabla u\|_{b,2}$ for all $V\geq0$.

Now for $\alpha>0$, which will be chosen later, we introduce a function
\[f(x)=\frac{1}{2}x^2-\frac{2-b}{\textbf{p}_c}\alpha^{2-\frac{\textbf{p}_c}{2-b}}x^{\frac{\textbf{p}_c}{2-b}} \quad\quad \mbox{for}\ x\in[0, \infty).\]
Differentiating the function $f$ with respect to $x$, we have
\[f'(x)=\alpha^{2-\frac{\textbf{p}_c}{2-b}}x(\alpha^{\frac{\textbf{p}_c}{2-b}-2}-x^{\frac{\textbf{p}_c}{2-b}-2}).\]
Since $\textbf{p}_c>2(2-b)$, we see that $f'(x)>0$ on $(0, \alpha)$ and $f'(x)<0$ on $(\alpha, \infty)$, which implies that $f(x)$ has a local maximum at $x_{\max}=\alpha$.

We choose $\alpha>0$ such that $\alpha=\|\nabla Q_{1}\|_{b,2}\|Q_{1}\|_{2}^\sigma$, then it follows from Pohozaev's identities \eqref{a3} that
\begin{equation}\label{c15}
f(\alpha)=\frac{\textbf{p}_c-2(2-b)}{2\textbf{p}_c}(\|\nabla Q_{1}\|_{b,2}\|Q_{1}\|_{2}^\sigma)^2=E_b(Q_{1})M(Q_{1})^\sigma.
\end{equation}
By hypothesis in Theorem \ref{cor1}, \eqref{c14}, \eqref{c15} and the conservation laws, we obtain
\[f(\|u\|_{\dot{H}_{b,V}^1}\|u\|_{2}^\sigma)\leq E_{b,V}(u)M(u)^\sigma <f(\|\nabla Q_{1}\|_{b,2}\|Q_{1}\|_{2}^\sigma)\quad \mbox{for}\ t\in[0, T^\ast).\]
Therefore, we see that
\begin{equation}\label{g1}
\mbox{either}\quad \|u\|_{\dot{H}_{b,V}^1}\|u\|_{2}^\sigma<\alpha\quad \mbox{or}\quad  \|u\|_{\dot{H}_{b,V}^1}\|u\|_{2}^\sigma>\alpha
\end{equation}
is valid for any $t\in[0, T^\ast)$.

Since $\|u_0\|_{\dot{H}_{b, V}^1}\|u_0\|_2^\sigma<\alpha$, by the continuity argument and \eqref{g1}, we have \[\|u(t)\|_{\dot{H}_{b, V}^1}\|u(t)\|_2^\sigma<\alpha\] for any $t\in[0, T^\ast)$, which proves the global existence in Theorem \ref{cor1} (i).

\subsection{Blow-up}
In this section, based on the virial identites, we prove the blow-up results for the intercritical INLS$_{b,V}$ with radial or finite variance initial data.

\textbf{Proof of Theorem \ref{thmm2} (ii).} Using \eqref{c4} and the virial identity in Corollary \ref{cor2}, we have
\begin{equation}\nonumber
\frac{d^2}{dt^2}\|u(t)\|^2_{2-b,2}=2(2-b)^2P(u(t)).
\end{equation}
Since $x\cdot\nabla V+(2-b)V\geq0$, we deduce that
\begin{eqnarray}\label{k1}
P(u(t))\leq \|\nabla u(t)\|^2_{b,2}+\int V(x)|u(t)|^2dx-\frac{\textbf{p}_c}{(2-b)(p+2)}\|u(t)\|^{p+2}_{c,p+2},
\end{eqnarray}
which implies that
\begin{equation}\nonumber
\frac{d^2}{dt^2}\|u(t)\|^2_{2-b,2}\leq 4(2-b)^2E_{b,V}(u(t))+\frac{2(2-b)^2}{p+2}(2-\frac{\textbf{p}_c}{2-b})\|u(t)\|^{p+2}_{c,p+2}.
\end{equation}
So by the energy conservation and $\textbf{p}_c=2(2-b)$, we have
\[\frac{d^2}{dt^2}\|u(t)\|^2_{2-b,2}\leq 4(2-b)^2E_{b,V}(u_0)<0.\]
Using the Glassey's argument (see \cite{GG}), we prove the finite time blow-up result for the fast-decaying and negative-energy solutions.

\textbf{Proof of Theorem \ref{cor1} (ii).} Firstly, we proceed as in the proof of Theorem \ref{cor1} (i), it follows from the assumption
\[\|u_0\|_{\dot{H}_{b, V}^1}\|u_0\|_2^\sigma>\|\nabla Q_{1}\|_{b,2}\|Q_{1}\|_2^\sigma\]
and \eqref{g1} that the lower bound \eqref{c13} of the solution $u(t)$ to the INLS$_{b,V}$ also holds.

\textbf{Case 1} ($u_0\in L^2(|x|^{2-b}dx)$). By Corollary \ref{cor2} and \eqref{c4}, we have
\begin{equation}\label{k3}
\frac{d^2}{dt^2}\|u(t)\|^2_{2-b,2}=2(2-b)^2P(u(t)).
\end{equation}
Now we claim that there exists $\delta=\delta(b, \textbf{p}_c, u_0)>0$ such that the virial functional bound
\begin{equation}\label{c16}
P(u)\leq-\delta
\end{equation}
holds for either \eqref{g42} or \eqref{b6}.

Indeed, from \eqref{k1}, it suffices to bound the integral term $\|u\|^2_{\dot{H}_{b,V}^1}-\frac{\textbf{p}_c}{(p+2)(2-b)}\|u\|^{p+2}_{c,p+2}$.
To our aim, we use the energy conservation to get
\begin{equation}\label{c18}
\|u\|^2_{\dot{H}_{b,V}^1}-\frac{\textbf{p}_c}{(p+2)(2-b)}\|u\|^{p+2}_{c,p+2}=\frac{\textbf{p}_c}{2-b}E_{b,V}(u_0)-(\frac{\textbf{p}_c}{4-2b}-1)\|u\|^2_{\dot{H}_{b,V}^1}.
\end{equation}
Since
\[E_{b,V}(u_0)<\frac{M(Q_{1})^\sigma}{M(u_0)^\sigma} E_b(Q_{1}),\]
let $\varepsilon=\varepsilon(b, \textbf{p}_c, u_0, Q_{1})>0$ be such that $\varepsilon:=\frac{1}{2}(\frac{M(Q_{1})^\sigma}{M(u_0)^\sigma}E_b(Q_{1})-E_{b,V}(u_0))$, we get
\begin{equation}\label{c19}
E_{b,V}(u_0)<\frac{1}{2}E_{b,V}(u_0)+\frac{1}{2}\frac{M(Q_{1})^\sigma}{M(u_0)^\sigma} E_b(Q_{1})=\frac{M(Q_{1})^\sigma}{M(u_0)^\sigma} E_b(Q_{1})-\varepsilon.
\end{equation}
Thus, collecting \eqref{c18}, \eqref{c19} and \eqref{c13}, we obtain
\begin{eqnarray}\nonumber
\lefteqn{\|u\|^2_{\dot{H}_{b,V}^1}-\frac{\textbf{p}_c}{(p+2)(2-b)}\|u\|^{p+2}_{c,p+2}}\\\nonumber
&&\leq\frac{\textbf{p}_c}{2-b}(\frac{M(Q_{1})^\sigma}{M(u_0)^\sigma} E_b(Q_{1})-\varepsilon)-(\frac{\textbf{p}_c}{4-2b}-1)\frac{\|Q_{1}\|_{2}^{2\sigma}}{\|u_0\|_{2}^{2\sigma}}\|\nabla Q_{1}\|^2_{b,2}\\\nonumber
&&=-\frac{\textbf{p}_c}{2-b}\varepsilon,
\end{eqnarray}
where the last equality is obtained by the fact that
\[E_b(Q_{1})=(\frac{1}{2}-\frac{2-b}{\textbf{p}_c})\|\nabla Q_{1}\|^2_{b,2}.\]
The claim \eqref{c16} follows.

At the end, combining \eqref{k3} with \eqref{c16}, we obtain that
\begin{equation}\nonumber
\frac{d^2}{dt^2}\|u(t)\|^2_{2-b,2}\leq-2(2-b)^2\delta
\end{equation}
for any $t\in[0, T^\ast)$.
This shows that the solution $u(t)$ to the INLS$_{b,V}$ blows up in finite time.

\textbf{Case 2} ($u_0$ is radial). The main blow-up mechanism is governed by an ODE equation $f^s<f',$ namely, if $s>1$, then $f$ will blow up in finite time. We proceed by contradiction and assume that the maximal existence $T^\ast=+\infty$.

To start it, we give a lower bound of blow-up rate. That is, there exist $C>0$ and $T_0>0$ such that
\begin{equation}\label{e26}
\|\nabla u(t)\|^2_{b,2}\geq Ct^2
\end{equation}
for all $t\geq T_0.$

In fact, from Lemma \ref{lem2}, if we take $\epsilon>0$ sufficiently small such that $C\epsilon<-b$ and then, we take  $R>0$ sufficiently large such that
\[C\|x\cdot\nabla V\|_{L^{\frac{n}{2}}(|x|>R;\ |x|^{-\frac{nb}{2}}dx)}<-b\]
and
\[C_{\epsilon}R^{-\frac{2\textbf{p}_c-(2-b)p}{4-p}}M(u_0)^{\frac{p}{4}+1}<2(2-b)\delta.\]
Thus we have
\[I''_{\psi_R}(t)<-2(2-b)\delta<0.\]

Now we integrate the above inequality from 0 to $t$, then
\[I'_{\psi_R}(t)\leq -2(2-b)\delta t+I'_{\psi_R}(0).\]
We may choose $T_0>0$ sufficiently large such that $I'_{\psi_R}(0)\leq (2-b)\delta T_0$. Then we get
\begin{equation}\label{e19}
I'_{\psi_R}(t)\leq -(2-b)\delta t,\quad\quad \forall\ t\geq T_0.
\end{equation}
Moreover, by Lemma \ref{lem2.1} and the H\"{o}lder inequality, we deduce
\begin{eqnarray}\nonumber
\lefteqn{|I'_{\psi_R}(t)|\ \leq2R\int|\nabla\Theta\left(\frac{x}{R}\right)||\nabla u||\overline{u}|dx}\\\nonumber
&&\quad\quad\leq 4R\int_{|x|\leq R}|\nabla u||\overline{u}|dx+CR\int_{R<|x|\leq 2R}|\nabla u||\overline{u}|dx\\\label{e20}
&&\quad\quad\leq CR^{1-\frac{b}{2}}\|\nabla u\|_{b,2}\|u\|_2,
\end{eqnarray}
where $\Theta$ is defined in \eqref{d13}.
Thus, combining \eqref{e19} with \eqref{e20}, we obtain
\[(2-b)\delta t\leq|I'_{\psi_R}(t)|\leq CR^{1-\frac{b}{2}}\|\nabla u\|_{b,2}M(u_0)^{\frac{1}{2}},\]
which arrives at \eqref{e26}.

To our aim, we also need to give another finer virial estimate about $I''_{\psi_R}(t)$, which is slight modification of \eqref{e23} in Lemma \ref{lem2}. Here we can see the difference between two cases in the treatment of the functional $P$ and the terms involving $\|\nabla u\|_{b,2}$.

By \eqref{e23}, \eqref{k1} and the energy conservation, we have
\begin{eqnarray}\nonumber
\lefteqn{I''_{\psi_R}(t)\leq4\textbf{p}_cE_{b,V}(u_0)-2(\textbf{p}_c-4+2b)(\|\nabla u\|^2_{b,2}+\int V(x)|u|^2dx)}\\\nonumber
&&\quad\ +C_{\epsilon}R^{-\frac{2\textbf{p}_c-(2-b)p}{4-p}}M(u_0)^{\frac{p}{4}+1}+CR^{b-2}M(u_0)\\\label{e22}
&&\quad\ +C\epsilon M(u_0)^{\frac{p}{4}+1}\|\nabla u\|_{b,2}^2+C\|x\cdot\nabla V\|_{L^{\frac{n}{2}}(|x|>R;\ |x|^{-\frac{nb}{2}}dx)}\|\nabla u\|_{b,2}^2.\quad\quad\quad\quad\quad
\end{eqnarray}
If we take $\epsilon>0$ sufficiently small such that $C\epsilon M(u_0)^{\frac{p}{4}+1}<\frac{1}{2}(\textbf{p}_c-4+2b)$ and take $R>0$ sufficiently large such that
\[C\|x\cdot\nabla V\|_{L^{\frac{n}{2}}(|x|>R;\ |x|^{-\frac{nb}{2}}dx)}<\frac{1}{2}(\textbf{p}_c-4+2b),\]
we then deduce from \eqref{e22} that
\begin{eqnarray}\nonumber
\lefteqn{I''_{\psi_R}(t)\leq4\textbf{p}_cE_{b,V}(u_0)-(\textbf{p}_c-4+2b)\|\nabla u\|^2_{b,2}}\\\nonumber
&&\quad\quad\quad\quad+C_{\epsilon}R^{-\frac{2\textbf{p}_c-(2-b)p}{4-p}}M(u_0)^{\frac{p}{4}+1}+CR^{b-2}M(u_0).
\end{eqnarray}
Thus, by \eqref{e26}, there exists $T_1\geq T_0$ and $\delta_0:=\frac{1}{2}(\textbf{p}_c-4+2b)>0$ such that
\begin{equation}\nonumber
I''_{\psi_R}(t)\leq-\delta_0\|\nabla u\|^2_{b,2}.
\end{equation}
We integrate the above inequality over $[T_1,\ t]$, then it follows from \eqref{e19} that
\begin{equation}\label{e25}
I'_{\psi_R}(t)\leq -\delta_0\int_{T_1}^t\|\nabla u(s)\|^2_{b,2}ds+I'_{\psi_R}(T_1)\leq -\delta_0\int_{T_1}^t\|\nabla u(s)\|^2_{b,2}ds.
\end{equation}
So, combining \eqref{e20} with \eqref{e25}, we obtain
\[\delta_0\int_{T_1}^t\|\nabla u(s)\|^2_{b,2}ds\leq |I'_{\psi_R}(t)|\leq CR^{1-\frac{b}{2}}M(u_0)^{\frac{1}{2}}\|\nabla u\|_{b,2}.\]
Define $f(t):=\int_{T_1}^t\|\nabla u(s)\|^2_{b,2}ds$ and thus
\[Af(t)^2\leq f'(t),\]
where $A:=\frac{\delta_0^2}{C^2R^{2-b}M(u_0)}$. Finally, taking $T_2>T_1$ and integrating on $[T_2,\ t]$, we get
\[A(t-T_2)\leq \frac{1}{f(T_2)}-\frac{1}{f(t)}\leq\frac{1}{f(T_2)}<\infty.\]
Due to $f(T_2)>0$, we let $t\rightarrow\infty$ and arrive at a contradiction. Therefore, the solution $u(t)$ to the INLS$_{b,V}$ blows up in finite time. We complete the proof of Theorem \ref{cor1} (ii).

\section{Variational Characterization}
In the first two subsections, we show the properties of minimizing problem \eqref{g6} for each $(\alpha, \beta)$ in \eqref{g2}, and give the proof of Proposition \ref{prop1}.
\subsection{Minimizing problems}
For each $(\alpha, \beta)\in\Bbb{R}^2$, we define the following functional:
\begin{equation}\label{g22}
\widetilde{J}_{\omega,0}^{\alpha, \beta}(\phi):=S_{\omega,0}(\phi)-\frac{K_{\omega,0}^{\alpha, \beta}(\phi)}{\mu},
\end{equation}
where $\mu:=\max\{2\alpha+(2-b-n)\beta,\ 2\alpha-n\beta\}$. 

\begin{lemma}\label{lemma5}
Let $\beta\in\Bbb{R},\ b<2$ and $\omega>0$. Then
\begin{equation}\nonumber
\mu\widetilde{J}_{\omega,0}^{\alpha, \beta}(\phi)\geq\frac{(2-b)|\beta|}{2}\min\{\|\nabla \phi\|_{b,2}^2,\ \omega\|\phi\|_2^2\}+\frac{p\alpha-(2-b+c)\beta+(2-b)\widetilde{\beta}}{p+2}\|\phi\|_{c,p+2}^{p+2},
\end{equation}
where
\begin{equation}\label{}
\widetilde{\beta}=\left\{
\begin{aligned}
 &\beta,\quad \mbox{if}\ \beta<0,\\\nonumber
 &0,\quad\ \mbox{if}\ \beta\geq0.
\end{aligned}\right.
\end{equation}
\end{lemma}
\begin{proof}
By the definition of $S_{\omega,0}$ and $K_{\omega,0}^{\alpha, \beta}$, we observe that
\begin{eqnarray}\nonumber
\lefteqn{\mu\widetilde{J}_{\omega,0}^{\alpha, \beta}(\phi)=\frac{1}{2}[\mu-2\alpha-(2-b-n)\beta]\|\nabla \phi\|_{b,2}^2+\frac{\omega}{2}[\mu-(2\alpha-n\beta)]\|\phi\|_2^2}\\\nonumber
&&\quad\quad\quad-\frac{1}{p+2}[\mu-(p+2)\alpha+(n+c)\beta]\|\phi\|_{c, p+2}^{p+2}.\quad\quad\quad\quad\quad\quad\quad
\end{eqnarray}
For any $\beta\in\Bbb{R}$, one can easily check that
\begin{equation}\label{}
[\mu-2\alpha-(2-b-n)\beta]\|\nabla \phi\|_{b,2}^2=\left\{
\begin{aligned}
 &-(2-b)\beta\|\nabla \phi\|_{b,2}^2,\quad \beta<0,\\\nonumber
 &0,\quad\quad\quad\quad\quad\quad\quad\quad\ \ \beta\geq0,
\end{aligned}\right.
\end{equation}
\begin{equation}\label{}
[\mu-(2\alpha-n\beta)]\|\phi\|_2^2=\left\{
\begin{aligned}
 & 0,\quad\quad\quad\quad\quad\quad \beta<0,\\\nonumber
 &(2-b)\beta\|\phi\|_2^2,\quad \beta\geq0,\quad\quad\quad\quad\quad\quad\quad\quad
\end{aligned}\right.
\end{equation}
\begin{equation}\label{}
[\mu-(p+2)\alpha+(n+c)\beta]\|\phi\|_{c, p+2}^{p+2}=\left\{
\begin{aligned}
 & (-p\alpha+c\beta)\|\phi\|_{c, p+2}^{p+2},\quad\quad\quad\quad\ \beta<0,\\\nonumber
 &[-p\alpha+(2-b+c)\beta]\|\phi\|_{c, p+2}^{p+2},\ \beta\geq0.
\end{aligned}\right.
\end{equation}
Therefore, we have
\begin{equation}\nonumber
\mu\widetilde{J}_{\omega,0}^{\alpha, \beta}(\phi)\geq\frac{(2-b)|\beta|}{2}\min\{\|\nabla \phi\|_{b,2}^2,\ \omega\|\phi\|_2^2\}+\frac{p\alpha-(2-b+c)\beta+(2-b)\widetilde{\beta}}{p+2}\|\phi\|_{c,p+2}^{p+2}.
\end{equation}
The proof of Lemma \ref{lemma5} is completed.
\end{proof}

Next, we prove the positivity of $K_{\omega,V}^{\alpha,\beta}$ near the origin in $W_b^{1,2}$.
\begin{lemma}\label{lemma1}
Let $2-n<b<2,\ b-2<c\leq0,\ 2(2-b)<\emph{\textbf{p}}_c\leq(2-b)(p+2),\ V\geq0,\ x\cdot\nabla V\leq0$ and let $(\alpha, \beta)$ satisfy \eqref{g2}. For $\omega>0$, if the uniform $L^2$-bounded sequence $\{\phi_n\}_{n\in\Bbb{N}}\in W_b^{1,2}\backslash\{0\}$ satisfies $\|\nabla\phi_n\|_{b,2}\rightarrow0$ as $n\rightarrow\infty$, then $K_{\omega,V}^{\alpha,\beta}(\phi_n)>0$ for sufficiently large $n$.
\end{lemma}

\begin{proof}
Since $V\geq0,\ x\cdot\nabla V\leq0$ and \eqref{g2}, we use the Gagliardo-Nirenberg inequality \eqref{g7} to get
\begin{eqnarray}\nonumber
\lefteqn{K_{\omega,V}^{\alpha,\beta}(\phi)\geq\frac{2\alpha+(2-b-n)\beta}{2}\|\nabla\phi\|^2_{b,2}-\frac{\alpha(p+2)-(n+c)\beta}{p+2}\|\phi\|^{p+2}_{c,p+2}}\\\nonumber
&&\quad\quad \geq \frac{2\alpha+(2-b-n)\beta}{2}\|\nabla\phi\|^2_{b,2}-\frac{\alpha(p+2)-(n+c)\beta}{(p+2)C}\|\nabla\phi\|^{\frac{\textbf{p}_c}{2-b}}_{b,2}\|\phi\|^{\frac{(2-b)(p+2)-\textbf{p}_c}{2-b}}_{2},
\end{eqnarray}
which is satisfied by the relation
\[\alpha(p+2)-(n+c)\beta\geq \alpha p-(2-b+c)\beta\geq0.\]

Since $\{\phi_n\}_{n\in\Bbb{N}}$ is bounded in $L^2$, there exists a constant $C_0>0$ such that
\[C_0:=\sup_{n\in\Bbb{N}}\|\phi_n\|_2<\infty.\]
By the hypothesis $\|\nabla\phi_n\|_{b,2}\rightarrow0$ as $n\rightarrow\infty$, there exists a large $N\in\Bbb{N}$ such that for $n\geq N$, the estimate holds
\[\|\nabla\phi_n\|_{b,2}\leq\left(\frac{2\alpha+(2-b-n)\beta}{4 CC_0^{p+2-\frac{\textbf{p}_c}{2-b}}}\frac{p+2}{\alpha(p+2)-(n+c)\beta}\right)^{\frac{2-b}{\textbf{p}_c-2(2-b)}}.\]
Therefore, combining the above three estimates, we obtain
\begin{eqnarray}\nonumber
\lefteqn{K_{\omega,V}^{\alpha,\beta}(\phi_n)\geq\|\nabla\phi_n\|^2_{b,2}}\\\nonumber
&&\quad\quad\quad \times\left(\frac{2\alpha+(2-b-n)\beta}{2}-\frac{\alpha(p+2)-(n+c)\beta}{p+2}CC_0^{p+2-\frac{\textbf{p}_c}{2-b}}\|\nabla\phi_n\|^{\frac{\textbf{p}_c}{2-b}-2}_{b,2}\right)\\\nonumber
&&\quad\quad\ \geq \frac{2\alpha+(2-b-n)\beta}{4}\|\nabla\phi_n\|^2_{b,2}>0,
\end{eqnarray}
which finishes the proof of this lemma.
\end{proof}

As a consequence, we claim that
\begin{equation}\label{g9}
K_{\omega,V}^{\alpha,\beta}(\phi_\lambda^{\alpha,\beta})>0
\end{equation}
for sufficiently small $\lambda<0$.
Indeed, from a simple computation
\[\|\nabla\phi_\lambda^{\alpha,\beta}\|_{b,2}=e^{\alpha\lambda+\frac{(2-b-n)\beta\lambda}{2}}\|\nabla\phi\|_{b,2},\]
we obtain that $\lim_{\lambda\rightarrow-\infty}\|\nabla\phi_\lambda^{\alpha,\beta}\|_{b,2}=0$, which together with Lemma \ref{lemma1} concludes \eqref{g9}.

Now we consider the minimizing problem $m_{\omega,0}^{\alpha,\beta}$ for each $(\alpha, \beta)$ in \eqref{g2}. To do it, we denote the set of nontrivial solutions to the equation \eqref{c1} by
\[\mathcal{A}_\omega:=\{\phi\in W_b^{1,2}\backslash\{0\}:\ S'_{\omega,0}(\phi)=0\}.\]
The set of ground states is denoted by $\mathcal{G}_\omega$:
\[\mathcal{G}_\omega:=\{\phi\in \mathcal{A}_\omega:\ S_{\omega,0}(\phi)\leq S_{\omega,0}(v)\quad \mbox{for\ all}\ v\in\mathcal{A}_\omega\}.\]
We can solve the minimizing problem \eqref{g6} in the case $V=0$.
\begin{pro}\label{pro2}
Let $n\geq3,\ 2-n<b<2,\ b-2< c\leq\min\{\frac{nb}{n-2},0\},\ 2(2-b)<\emph{\textbf{p}}_c<(2-b)(p+2)$ and let $(\alpha, \beta)$ satisfy \eqref{g2}. Then for $\omega>0$, there exists $\phi\in W_b^{1,2}\backslash\{0\}$ such that
\[S_{\omega,0}(\phi)=m_{\omega,0}^{\alpha,\beta}\quad\  \mbox{and}\quad\  K_{\omega,0}^{\alpha, \beta}(\phi)=0.\]

\end{pro}

\begin{proof}
We split the proof of Proposition \ref{pro2} into the following three steps.

\textbf{Step 1.} We first prove that the minimizing sequence of \eqref{g6} with $V=0$ is bounded in $W_b^{1,2}$ and the weak limit of it is nonzero.

Let $\{\phi_n\}_{n\in\Bbb{N}}\subset W_b^{1,2}\backslash\{0\}$ be a minimizing sequence of \eqref{g6} with $V=0$, namely
\[K_{\omega,0}^{\alpha, \beta}(\phi_n)=0\quad \mbox{and}\quad S_{\omega,0}(\phi_n)\rightarrow m_{\omega,0}^{\alpha,\beta}\quad\quad \mbox{as}\quad n\rightarrow\infty.\]
Thus it follows from \eqref{g22} that
\begin{equation}\label{b5}
\widetilde{J}_{\omega,0}^{\alpha, \beta}(\phi_n)=S_{\omega,0}(\phi_n)\rightarrow m_{\omega,0}^{\alpha,\beta}\quad\quad \mbox{as}\quad  n\rightarrow\infty.
\end{equation}
Since $(\alpha, \beta)$ satisfies \eqref{g2} and $\textbf{p}_c>2(2-b)$, Lemma \ref{lemma5} and \eqref{b5} yield that $\|\phi_n\|_{c,p+2}^{p+2}$ is bounded. Since $K_{\omega,0}^{\alpha, \beta}(\phi_n)=0$, then
\[\frac{2\alpha+(2-b-n)\beta}{2}\|\nabla\phi_n\|_{b,2}^2+\frac{(2\alpha-n\beta)\omega}{2}\|\phi_n\|_2^2\leq C\]
for some $C>0$, which shows that $\{\phi_n\}_{n\in\Bbb{N}}$ is bounded in $W^{1,2}_b$.
Therefore there exists $\phi\in W_b^{1,2}$ such that, up to a subsequence, $\phi_n\rightharpoonup \phi$ weakly in $W^{1,2}_b$. From Lemma \ref{lemma6}, we obtain
\begin{equation}\label{b29}
\phi_n\rightarrow \phi\quad \mbox{in}\quad L^{p+2}(|x|^cdx)
\end{equation}
for all $2(2-b)<\textbf{p}_c<(2-b)(p+2)$.

Now we claim that $\phi\neq0$. Assume that $\phi=0$ by contradiction. Then we deduce from $K_{\omega,0}^{\alpha, \beta}(\phi_n)=0$ and \eqref{b29} that
\[\frac{2\alpha+(2-b-n)\beta}{2}\|\nabla\phi_n\|_{b,2}^2+\frac{(2\alpha-n\beta)\omega}{2}\|\phi_n\|_2^2\rightarrow0\quad\quad \mbox{as}\quad  n\rightarrow\infty.\]
Thus, Lemma \ref{lemma1} implies that $K_{\omega,0}^{\alpha,\beta}(\phi_n)>0$ for sufficiently large $n$, which contradicts $K_{\omega,0}^{\alpha,\beta}(\phi_n)=0$. Hence $\phi\neq0$.

\textbf{Step 3.} We prove that $\phi$ is the minimizer of \eqref{g6} with $V=0$.

Denote $\phi_n=\phi+r_n$, where $r_n\rightharpoonup 0$ in $W_b^{1,2}$ as $n\rightarrow\infty$. By the definition of $K_{\omega,0}^{\alpha, \beta}$ and a simple computation, we have
\begin{equation}\nonumber
K_{\omega,0}^{\alpha, \beta}(\phi)=K_{\omega,0}^{\alpha, \beta}(\phi_n)+\frac{2\alpha+(2-b-n)\beta}{2}\|\nabla r_n\|_{b,2}^2+\frac{(2\alpha-n\beta)\omega}{2}\|r_n\|_2^2+o_n(1),
\end{equation}
where $o_n(1)\rightarrow0$ as $n\rightarrow\infty.$ This implies that
\[K_{\omega,0}^{\alpha, \beta}(\phi)\leq \liminf_{n\rightarrow\infty}K_{\omega,0}^{\alpha, \beta}(\phi_n)=0.\]
So $K_{\omega,0}^{\alpha, \beta}(\phi)\leq0.$

Now we assume that $K_{\omega,0}^{\alpha, \beta}(\phi)<0$. From Lemma \ref{lemma1}, we easily deduce that $K_{\omega,0}^{\alpha,\beta}(\lambda\phi)>0$ for sufficiently small $\lambda\in(0, 1)$. This together with the continuity of $K_{\omega,0}^{\alpha,\beta}(\lambda\phi)$ in $\lambda$ implies that there exists $\lambda_0\in(0,1)$ such that $K_{\omega,0}^{\alpha, \beta}(\lambda_0\phi)=0$.
By the definition of $m_{\omega,0}^{\alpha, \beta}$ and \eqref{g22}, we have
\begin{equation}\label{g29}
m_{\omega,0}^{\alpha, \beta}\leq S_{\omega,0}(\lambda_0\phi)=\widetilde{J}_{\omega,0}^{\alpha, \beta}(\lambda_0\phi)<\lambda_0^2\widetilde{J}_{\omega,0}^{\alpha, \beta}(\phi).
\end{equation}
Moreover, by a simple computation, we obtain
\[\widetilde{J}_{\omega,0}^{\alpha, \beta}(\phi)=\widetilde{J}_{\omega,0}^{\alpha, \beta}(\phi_n)+o_n(1),\]
where $o_n(1)\rightarrow0$ as $n\rightarrow\infty$. This together with \eqref{b5} yields
\begin{equation}\label{g30}
\widetilde{J}_{\omega,0}^{\alpha, \beta}(\phi)\leq \liminf_{n\rightarrow\infty}\widetilde{J}_{\omega,0}^{\alpha, \beta}(\phi_n)\rightarrow m_{\omega,0}^{\alpha, \beta}.
\end{equation}
Therefore, combining \eqref{g29} with \eqref{g30}, we infer that
\[m_{\omega,0}^{\alpha, \beta}\leq S_{\omega,0}(\lambda_0\phi)<\lambda_0^2\widetilde{J}_{\omega,0}^{\alpha, \beta}(\phi)\leq m_{\omega,0}^{\alpha, \beta},\]
which is a contradiction. Hence, $K_{\omega,0}^{\alpha, \beta}(\phi)=0$. Moreover,
\[S_{\omega,0}(\phi)=\widetilde{J}_{\omega,0}^{\alpha, \beta}(\phi)\leq m_{\omega,0}^{\alpha, \beta},\]
which together with the definition of $m_{\omega,0}^{\alpha, \beta}$ yields that $S_{\omega,0}(\phi)=m_{\omega,0}^{\alpha, \beta}$. Hence, $\phi$ is the minimizer of \eqref{g6} with $V=0$.
\end{proof}

\subsection{Proof of Proposition \ref{prop1}}
Now, we are able to prove Proposition \ref{prop1}, which shows the parameter independence of $m_{\omega,0}^{\alpha, \beta}$ via its characterization by the ground states $Q_{\omega,0}$.

\textbf{Proof of Proposition \ref{prop1}.}
(i) We consider the minimization problem
\begin{equation}
\mathcal{M}_\omega=\{\phi\in W_b^{1,2}:\ S_{\omega,0}(\phi)=m_{\omega,0}^{\alpha,\beta},\ \ K_{\omega,0}^{\alpha,\beta}(\phi)=0\}.
\end{equation}
By Proposition \ref{pro2}, we show that $\mathcal{M}_\omega\neq\emptyset$. To show that the minimizer of \eqref{g6} with $V=0$ is also the ground state of the equation \eqref{c1}, it suffices to prove that $\mathcal{M}_\omega=\mathcal{G}_\omega$. 

On one hand, we prove $\mathcal{M}_\omega\subset\mathcal{G}_\omega$. Assume $\phi\in\mathcal{M}_{\omega}$, we firstly claim that $\phi\in\mathcal{A}_\omega$.

Indeed, since $\phi$ is the minimizer of \eqref{g6} with $V=0$, so there exists a Lagrange multiplier $\tau\in\Bbb{R}$ such that $S'_{\omega,0}(\phi)=\tau K'_{\omega,0}(\phi)$.
Recalling the definition $K_{\omega,0}$, we deduce
\begin{equation}\label{g31}
0=K_{\omega,0}(\phi)=\mathcal{L}_{\alpha, \beta}S_{\omega,0}(\phi)=\langle S'_{\omega,0}(\phi),\ \mathcal{L}_{\alpha, \beta}\phi \rangle=\tau \langle K'_{\omega,0}(\phi),\ \mathcal{L}_{\alpha, \beta}\phi \rangle,
\end{equation}
where we write $K_{\omega,0}:=K_{\omega,0}^{\alpha, \beta}$ and $\mathcal{L}_{\alpha, \beta}\phi:=\frac{\partial}{\partial\lambda}\phi_\lambda^{\alpha, \beta}|_{\lambda=0}$ for simplicity.

Since
\[\mathcal{L}_{\alpha, \beta}K_{\omega,0}(\phi)=\langle K'_{\omega,0}(\phi),\ \mathcal{L}_{\alpha, \beta}\phi \rangle,\]
which linked with \eqref{g31} yields that
\begin{equation}\label{g33}
0=K_{\omega,0}(\phi)=\tau\mathcal{L}_{\alpha, \beta}K_{\omega,0}(\phi)=\tau\mathcal{L}_{\alpha, \beta}^2S_{\omega,0}(\phi).
\end{equation}
We observe that
\begin{eqnarray}\nonumber
\lefteqn{\mathcal{L}_{\alpha, \beta}^2S_{\omega,0}(\phi)=\frac{[2\alpha+(2-b-n)\beta]^2}{2}\|\nabla\phi\|_{b,2}^2}\\\nonumber
&&\quad\quad\quad\quad+\frac{\omega(2\alpha-n\beta)^2}{2}\|\phi\|_2^2-\frac{[\alpha(p+2)-(n+c)\beta]^2}{p+2}\|\phi\|_{c,p+2}^{p+2}\\\nonumber
&&\quad\quad\quad\leq\mu\left(\frac{2\alpha+(2-b-n)\beta}{2}\|\nabla\phi\|_{b,2}^2+\frac{\omega(2\alpha-n\beta)}{2}\|\phi\|_2^2\right)\\\nonumber
&&\quad\quad\quad\quad-\frac{[\alpha(p+2)-(n+c)\beta]^2}{p+2}\|\phi\|_{c,p+2}^{p+2}\\\label{g32}
&&\quad\quad\quad=\frac{\alpha(p+2)-(n+c)\beta}{p+2}[\mu-\alpha(p+2)+(n+c)\beta]\|\phi\|_{c,p+2}^{p+2}<0,\quad\quad\quad
\end{eqnarray}
where the last inequality holds by the conditions that $c\geq b-2,\ \textbf{p}_c>2(2-b)$ and \eqref{g2}. Thus, it follows from \eqref{g33} and \eqref{g32} that
\[0=K_{\omega,0}(\phi)=\tau\mathcal{L}_{\alpha, \beta}^2S_{\omega,0}(\phi)<0.\]
This implies that $\tau=0$, and hence $S'_{\omega,0}(\phi)=0$. Therefore $\phi\in\mathcal{A}_\omega$ is proved.

Next, for any $v\in\mathcal{A}_\omega$, we deduce from \eqref{g31} that $K_{\omega,0}(v)=\langle S'_{\omega,0}(v),\ \mathcal{L}_{\alpha, \beta}v \rangle=0$ and $v\neq0$. By the definition of $\mathcal{M}_{\omega}$, we have
\[S_{\omega,0}(\phi)\leq S_{\omega,0}(v).\]
This implies that $\phi\in\mathcal{G}_{\omega}$. So $\mathcal{M}_\omega\subset\mathcal{G}_\omega$.

On the other hand, from Proposition \ref{pro2}, there exists $v\in\mathcal{M}_{\omega}$ such that $S_{\omega,0}(v)=m_{\omega,0}^{\alpha,\beta}$ and $K_{\omega,0}(v)=0,$ which immediately implies $v\in\mathcal{A}_{\omega}$. So if we assume $\phi\in\mathcal{G}_{\omega}$,
by the definition of $\mathcal{G}_{\omega}$, we get
\[S_{\omega,0}(\phi)\leq S_{\omega,0}(v)=m_{\omega,0}^{\alpha,\beta}.\]
Moreover, we deduce from $\phi\in\mathcal{A}_\omega$ and \eqref{g31} that $K_{\omega,0}(\phi)=0$ and $\phi\neq0$, which implies that $m_{\omega,0}^{\alpha,\beta}\leq S_{\omega,0}(\phi)$. Thus, $S_{\omega,0}(\phi)=m_{\omega,0}^{\alpha,\beta}$ and then $\mathcal{G}_{\omega}\subset\mathcal{M}_{\omega}$. Therefore, we obtain that $\mathcal{G}_{\omega}=\mathcal{M}_{\omega}$. Hence, $m_{\omega,0}^{\alpha,\beta}$ is attained by the ground state $Q_\omega$ of \eqref{c1}.

(ii) We introduce the functional
\begin{eqnarray}\label{g8}
\lefteqn{J_{\omega,V}^{\alpha, \beta}(\phi):=S_{\omega,V}(\phi)-\frac{1}{2\alpha+(2-b-n)\beta}K_{\omega,V}^{\alpha, \beta}(\phi)}\\\nonumber
&&\quad\quad=\frac{\beta}{4\alpha+2(2-b-n)\beta}\int[(2-b)V+x\cdot\nabla V+(2-b)\omega]|\phi|^2dx\\\nonumber
&&\quad\quad\quad+\frac{\alpha p-(2-b+c)\beta}{(p+2)[2\alpha+(2-b-n)\beta]}\|\phi\|_{c,p+2}^{p+2}.
\end{eqnarray}
In particular, $J_{\omega,0}^{\alpha, \beta}=\widetilde{J}_{\omega,0}^{\alpha, \beta}$ in case $V=0$ and $\beta\geq0$, which is defined in \eqref{g22}.

By considering the following minimizing problem in terms of $J_{\omega,V}^{\alpha, \beta}$:
\[n_{\omega,V}^{\alpha, \beta}:=\inf\{J_{\omega,V}^{\alpha, \beta}(\phi):\ \phi\in W_b^{1,2}\backslash\{0\},\ K_{\omega,V}^{\alpha, \beta}(\phi)\leq0\},\]
we claim that
\begin{equation}\label{g36}
n_{\omega,V}^{\alpha, \beta}=m_{\omega,V}^{\alpha,\beta}.
\end{equation}

Indeed, Let $\phi\in W_b^{1,2}\backslash\{0\}$ satisfy $K_{\omega,V}^{\alpha, \beta}(\phi)\leq0$. If $K_{\omega,V}^{\alpha, \beta}(\phi)=0$, by the definition of $m_{\omega,V}^{\alpha,\beta}$ and \eqref{g8}, we obtain
\[m_{\omega,V}^{\alpha,\beta}\leq S_{\omega,V}(\phi)=J_{\omega,V}^{\alpha, \beta}(\phi).\]
If else $K_{\omega,V}^{\alpha, \beta}(\phi)<0$. From Lemma \ref{lemma1} and the continuity of $K_{\omega,V}^{\alpha,\beta}(\lambda\phi)$ in $\lambda$, there exists $\lambda_0\in(0, 1)$ such that $K_{\omega,V}^{\alpha,\beta}(\lambda_0\phi)=0$. Thus,
\[m_{\omega,V}^{\alpha,\beta}\leq S_{\omega,V}(\lambda_0\phi)=J_{\omega,V}^{\alpha, \beta}(\lambda_0\phi)<J_{\omega,V}^{\alpha, \beta}(\phi),\]
where the last inequality is obtained by the fact that $J_{\omega,V}^{\alpha, \beta}(\lambda\phi)$ is monotone increasing in $\lambda\in(0, 1)$.
Hence, we conclude that $n_{\omega,V}^{\alpha, \beta}\geq m_{\omega,V}^{\alpha,\beta}$.

To prove  $n_{\omega,V}^{\alpha, \beta}\leq m_{\omega,V}^{\alpha,\beta}$, we assume $\phi\in W_b^{1,2}\backslash\{0\}$ satisfies $K_{\omega,V}^{\alpha, \beta}(\phi)=0$. By the definition of $n_{\omega,V}^{\alpha, \beta}$ and \eqref{g8}, we get $n_{\omega,V}^{\alpha, \beta}\leq J_{\omega,V}^{\alpha, \beta}(\phi)=S_{\omega,V}$. Thus we have $n_{\omega,V}^{\alpha, \beta}\leq m_{\omega,V}^{\alpha,\beta}$. Therefore, we conclude \eqref{g36}.

Now, let $\phi\in W_b^{1,2}\backslash\{0\}$ satisfies $K_{\omega,V}^{\alpha, \beta}(\phi)=0$. Then we obtain $S_{\omega,V}(\phi)=J_{\omega,V}^{\alpha, \beta}(\phi)$. Since $V\geq0$ and $x\cdot\nabla V\leq0$, we get $K_{\omega,0}^{\alpha, \beta}(\phi)\leq K_{\omega,V}^{\alpha, \beta}(\phi)=0$. Thus it follows from \eqref{g36} that
\[m_{\omega,0}^{\alpha,\beta}\leq J_{\omega,0}^{\alpha,\beta}(\phi)\leq J_{\omega,V}^{\alpha,\beta}(\phi)=S_{\omega,V}(\phi),\]
which arrives at $m_{\omega,0}^{\alpha,\beta}\leq m_{\omega,V}^{\alpha,\beta}$. Hence, we finish the proof of Proposition \ref{prop1}. \quad\quad\quad $\square$

\subsection{Invariant Flow}
In this part, we will prove that the sets $\mathcal{N}^{\pm}$ defined in \eqref{g23} and \eqref{g24} are invariant under the flow of \eqref{a0}.

We denote
\[L_{\omega,V}(\phi):=\|\phi\|_{\dot{H}_{b,V}^1}+\omega\|\phi\|_2^2\]
and the frequency
\[\omega_1:=-\frac{1}{2}\inf_{x\in\Bbb{R}^n}[(2-b)V+x\cdot\nabla V].\]
The following lemma makes a comparison between $S_{\omega,V}(\phi)$ and $L_{\omega,V}(\phi)$.
\begin{lemma}\label{lemma3}
Let $b<2,\ c\geq b-2,\ 2(2-b)<\emph{\textbf{p}}_c\leq(2-b)(p+2),\ x\cdot\nabla V+(2-b)V\geq0$ and let $(\alpha,\beta)$ satisfy \eqref{g2}. Assume $\omega\geq\omega_1$ and $\phi\in W_b^{1,2}$ satisfies $K_{\omega,V}^{\alpha,\beta}(\phi)\geq0$. Then
\begin{equation}\nonumber
2S_{\omega,V}(\phi)\leq L_{\omega,V}(\phi)\leq\frac{2\alpha(p+2)-2(n+c)\beta}{\alpha p-(2-b+c)\beta}S_{\omega,V}(\phi).
\end{equation}
\end{lemma}

\begin{proof}
Since
\[S_{\omega,V}(\phi)=\frac{1}{2}L_{\omega,V}(\phi)-\frac{1}{p+2}\|\phi\|_{c,p+2}^{p+2}\leq\frac{1}{2}L_{\omega,V}(\phi),\]
it is clear that the first inequality holds. To prove the second inequality, we rewrite \eqref{g3} as
\begin{eqnarray}\nonumber
\lefteqn{K_{\omega,V}^{\alpha,\beta}(\phi)=\frac{2\alpha+(2-b-n)\beta}{2}L_{\omega,V}(\phi)-\frac{\alpha(p+2)-(n+c)\beta}{p+2}\|\phi\|^{p+2}_{c,p+2}}\\\nonumber
&&\quad\quad\quad -\frac{\beta}{2}\int [(2-b)V+x\cdot\nabla V+(2-b)\omega]|\phi|^2dx,\quad\quad\quad\quad\quad\quad
\end{eqnarray}
then we further have
\begin{eqnarray}\nonumber
\lefteqn{K_{\omega,V}^{\alpha,\beta}(\phi)+\frac{\alpha p-(2-b+c)\beta}{2}L_{\omega,V}(\phi)=[\alpha(p+2)-(n+c)\beta]S_{\omega,V}(\phi)}\\\nonumber
&&\quad\quad\quad\quad\quad\quad\quad\quad\quad\quad-\frac{\beta}{2}\int [(2-b)V+x\cdot\nabla V+(2-b)\omega]|\phi|^2dx\\\nonumber
&&\quad\quad\quad\quad\quad\quad\quad\quad\quad\leq [\alpha(p+2)-(n+c)\beta]S_{\omega,V}(\phi).\quad\quad\quad\quad\quad\quad\quad\quad
\end{eqnarray}
Thus, together with $K_{\omega,V}^{\alpha,\beta}(\phi)\geq0$, we infer that
\[\frac{\alpha p-(2-b+c)\beta}{2}L_{\omega,V}(\phi)\leq [\alpha(p+2)-(n+c)\beta]S_{\omega,V}(\phi).\]
We prove this lemma.

\end{proof}

\begin{lemma}\label{lemma4}
Let $2-n<b<2,\ b-2<c\leq0,\ 2(2-b)<\emph{\textbf{p}}_c\leq(2-b)(p+2),\ V\geq0, -(2-b)V\leq x\cdot\nabla V\leq0,\ x\nabla^2Vx^T\leq-(n+3-b-\frac{2\alpha}{\beta})x\cdot\nabla V$ and let $(\alpha, \beta)$ satisfy \eqref{g2}.
Assume $\phi\in \mathcal{N}_{\alpha,\beta}:=\{\phi\in W_b^{1,2}:\ S_{\omega,V}(\phi)<m_{\omega,0},\ K_{\omega,V}^{\alpha,\beta}(\phi)<0\}$. Then
\begin{equation}\label{g14}
K_{\omega,V}^{\alpha, \beta}(\phi)\leq-[2\alpha+(2-b-n)\beta](m_{\omega,0}-S_{\omega,V}(\phi)),
\end{equation}
where $\nabla^2V$ is the Hessian matrix of $V$.
\end{lemma}
\begin{proof}
We write $g(\lambda):=S_{\omega,V}(\phi_\lambda^{\alpha, \beta})$, where $\phi_\lambda^{\alpha, \beta}(x)=e^{\alpha\lambda}\phi(e^{\beta\lambda}x)$. Then
\begin{eqnarray}\nonumber
\lefteqn{g'(\lambda)=\frac{2\alpha+(2-b-n)\beta}{2}e^{[2\alpha+(2-b-n)\beta]\lambda}\|\nabla\phi\|^2_{b,2}}\\\nonumber
&&\quad+\frac{2\alpha-n\beta}{2}e^{(2\alpha-n\beta)\lambda}(\omega\|\phi\|_2^2+\int V(e^{-\beta\lambda}x)|\phi|^2dx)\\\nonumber
&&\quad-\frac{\beta e^{[2\alpha-(n+1)\beta]\lambda}}{2}\int x\cdot\nabla V(e^{-\beta\lambda}x)|\phi|^2dx\\\label{g37}
&&\quad-\frac{(p+2)\alpha-(n+c)\beta}{p+2}e^{[(p+2)\alpha-(n+c)\beta]\lambda}\|\phi\|^{p+2}_{c,p+2}.\quad\quad\quad
\end{eqnarray}

Since $\phi\in\mathcal{N}_{\alpha,\beta}$, we deduce that
\begin{equation}\nonumber
g'(0)=K_{\omega,V}^{\alpha,\beta}(\phi)<0.
\end{equation}
By \eqref{g9}, we get
\begin{equation}\nonumber
K_{\omega,V}^{\alpha,\beta}(\phi_\lambda^{\alpha,\beta})=g'(\lambda)>0
\end{equation}
for sufficiently small $\lambda<0$. Thus by the continuity of $g'(\lambda)$ in $\lambda$, there exists $\lambda_0<0$ such that
\begin{equation}\label{g25}
g'(\lambda_0)=0 \quad \mbox{and}\quad g'(\lambda)<0\quad \mbox{for}\ \lambda\in(\lambda_0, 0].
\end{equation}
Thus, by the definition of $m_{\omega,V}^{\alpha,\beta}$ and Proposition \ref{prop1} (ii), we have
\begin{equation}\label{g13}
m_{\omega,0}\leq m_{\omega,V}^{\alpha,\beta}\leq S_{\omega,V}(\phi_{\lambda_0}^{\alpha,\beta})=g(\lambda_0).
\end{equation}

Now we differentiate $g'(\lambda)$ again to obtain
\begin{eqnarray}\nonumber
\lefteqn{g''(\lambda)=\frac{[2\alpha+(2-b-n)\beta]^2}{2}e^{[2\alpha+(2-b-n)\beta]\lambda}\|\nabla\phi\|^2_{b,2}}\\\nonumber
&&\quad\quad+\frac{(2\alpha-n\beta)^2}{2}e^{(2\alpha-n\beta)\lambda}(\omega\|\phi\|_2^2+\int V(e^{-\beta\lambda}x)|\phi|^2dx)\\\nonumber
&&\quad\quad+\frac{(2\alpha-n\beta)^2}{2}e^{(2\alpha-n\beta)\lambda}\int x\cdot\nabla V(e^{-\beta\lambda}x)|\phi|^2dx\\\nonumber
&&\quad\quad+\frac{\beta^2}{2}e^{[2\alpha-(n+2)\beta]\lambda}\int x\nabla^2 V(e^{-\beta\lambda}x)x^T|\phi|^2dx\\\nonumber
&&\quad\quad-\frac{[(p+2)\alpha-(n+c)\beta]^2}{p+2}e^{[(p+2)\alpha-(n+c)\beta]\lambda}\|\phi\|^{p+2}_{c,p+2}\\\nonumber
&&\quad\leq\frac{[2\alpha+(2-b-n)\beta]^2}{2}e^{[2\alpha+(2-b-n)\beta]\lambda}\|\nabla\phi\|^2_{b,2}\\\nonumber
&&\quad\quad+\frac{(2\alpha-n\beta)^2}{2}e^{(2\alpha-n\beta)\lambda}(\omega\|\phi\|_2^2+\int V(e^{-\beta\lambda}x)|\phi|^2dx)\\\nonumber
&&\quad\quad-\frac{\beta[2\alpha+(2-b-n)\beta]}{2}e^{[2\alpha-(n+1)\beta]\lambda}\int x\cdot\nabla V(e^{-\beta\lambda}x)|\phi|^2dx\\\nonumber
&&\quad\quad-\frac{[(p+2)\alpha-(n+c)\beta]^2}{p+2}e^{[(p+2)\alpha-(n+c)\beta]\lambda}\|\phi\|^{p+2}_{c,p+2},
\end{eqnarray}
which together with \eqref{g37} yields
\begin{eqnarray}\nonumber
\lefteqn{g''(\lambda)-[2\alpha+(2-b-n)\beta]g'(\lambda)}\\\nonumber
&&\leq-\frac{(2-b)(2\alpha-n\beta)\beta}{2}e^{(2\alpha-n\beta)\lambda}(\omega\|\phi\|_2^2+\int V(e^{-\beta\lambda}x)|\phi|^2dx)\\\nonumber
&&\ -\frac{[(p+2)\alpha-(n+c)\beta][\alpha p-(2-b+c)\beta]}{p+2}e^{[(p+2)\alpha-(n+c)\beta]\lambda}\|\phi\|^{p+2}_{c,p+2}\leq0.\quad\quad\quad
\end{eqnarray}
Integrating the above inequality over $[\lambda_0, 0]$, we get
\[g'(0)-g'(\lambda_0)\leq [2\alpha+(2-b-n)\beta](g(0)-g(\lambda_0)).\]
Thus, by \eqref{g25} and \eqref{g13}, we have $K_{\omega,V}^{\alpha, \beta}(\phi)\leq-[2\alpha+(2-b-n)\beta](m_{\omega,0}-S_{\omega,V}(\phi))$. Hence, Lemma \ref{lemma4} is proved.
\end{proof}
\begin{pro}\label{lemm5}
Let $2-n<b<2,\ b-2<c\leq0,\ 2(2-b)<\emph{\textbf{p}}_c<(2-b)(p+2),\ V\geq0, -(2-b)V\leq x\cdot\nabla V\leq0,\ x\nabla^2Vx^T\leq-(3-b)x\cdot\nabla V$. Then the sets $\mathcal{N}^{\pm}$ are invariant under the flow of \eqref{a0}. 
\end{pro}
\begin{proof}
Since $u_0\in \mathcal{N}^+$, then $S_{\omega,V}(u_0)<m_{\omega,0}$ and $K_{\omega,V}^{n,2}(u_0)\geq0$. We discuss it into two cases:

When $K_{\omega,V}^{n,2}(u_0)=0$, if follows from the definition of $m_{\omega,V}^{n,2}$ and Proposition \ref{prop1} (ii) that
\[m_{\omega,V}^{n,2}\leq S_{\omega,V}(u_0)<m_{\omega,0}\leq m_{\omega,V}^{n,2}.\]
Thus we have $u_0\equiv0$, and then $u(x,t)\equiv0$. Therefore, $u(t)\in \mathcal{N}^+$. 

When $K_{\omega,V}^{n,2}(u_0)>0$, we prove $u(t)\in \mathcal{N}^+$ by contradiction as follows. If the conclusion does not hold, then there exists $t_0\in[0, T^\ast)$ such that $K_{\omega,V}^{n,2}(u(t_0))=0$. This implies that
\[m_{\omega,V}^{n,2}\leq S_{\omega,V}(u(t_0)).\]
By Proposition \ref{prop1} and the conservation laws, we infer that
\[S_{\omega,V}(u(t_0))=S_{\omega,V}(u_0)< m_{\omega,0}\leq m_{\omega,V}^{n,2},\]
which arrives at a contradiction. Hence $u(t)\in \mathcal{N}^+$.

For $u_0\in \mathcal{N}^-$, Lemma \ref{lemma4} and the conservation laws imply that $K_{\omega,V}^{n,2}(u)\leq-2(2-b)(m_{\omega,0}-S_{\omega,V}(u_0))<0$, which concludes $u(t)\in \mathcal{N}^-$.

\end{proof}

\section{Global existence and Blow-up for non-radial case}
In this section, we prove Theorem \ref{thmm1} and Corollary \ref{thm4}, of which the proof of blow-up relies on the method of Du-Wu-Zhang \cite{DWZ}. 
To start it, we establish the following $L^2$ estimate in the exterior ball.
\begin{lemma}\label{lem8}
For $2-n<b\leq0$. Let $u\in C([0, \infty); W^{1,2}_b)$ be the solution to the INLS$_{b,V}$, which satisfies
\begin{equation}\label{d4}
\sup_{t\in[0, \infty)}\|\nabla u(t)\|_{b,2}<\infty.
\end{equation}
Then there exists a constant $C>0$ such that for any $\eta>0,\ R>0$ and $t\in[0,\ \frac{\eta R^{1-b/2}}{C_1C\|u_0\|_2}]$, the following estimate
\begin{equation}\label{d5}
\int_{|x|>R}|u(x,t)|^2dx\leq o_R(1)+\eta
\end{equation}
holds, where $o_R(1)$ denotes a function of $R$ satisfying $o_R(1)\rightarrow0$ as $R\rightarrow\infty$.
\end{lemma}
\textbf{Proof of Lemma \ref{lem8}.} Let $\psi_R(x)=\theta(\frac{x}{R})$ be a radial function defined as
\begin{equation}\theta(x)=
\left\{
\begin{aligned}
 & 0,\quad\quad\quad\quad \mbox{if}\quad 0\leq |x|\leq \frac{1}{2},\\\nonumber
 & \mbox{smooth},\quad\ \mbox{if}\quad \frac{1}{2}<|x|<1,\\\nonumber
 & 1,\quad\quad\quad\quad  \mbox{if}\quad 1\leq|x|.
\end{aligned}\right.
\end{equation}
We note that there exists constant $C>0$ such that $|\nabla\theta|\leq C$.
Using \eqref{b26} in Lemma \ref{lem2.1}, $\nabla\psi_R=\frac{\nabla\phi_R}{|x|^b}$ and \eqref{d4}, we have
\begin{eqnarray}\nonumber
\lefteqn{I_{\psi_R}(t)=I_{\psi_R}(0)+\int_0^t\frac{d}{ds}I_{\psi_R}(s)ds}\\\nonumber
&&\quad\leq I_{\psi_R}(0)+\int_0^t |2\textmd{Im}\int_{|x|\geq\frac{R}{2}}|x|^b\nabla\psi_R\cdot\nabla u\overline{u}dx|ds\\\nonumber
&&\quad\leq I_{\psi_R}(0)+2t\||x|^{\frac{b}{2}}\nabla\psi_R\|_{L^\infty(|x|\geq\frac{R}{2})}\|\nabla u\|_{b,2}\|u\|_2\\\nonumber
&&\quad\leq I_{\psi_R}(0)+\frac{C_1C\|u_0\|_2t}{R^{1-\frac{b}{2}}}
\end{eqnarray}
for any $t\in [0, \infty)$, where $C_1:=\sup_{t\in[0, \infty)}\|\nabla u(t)\|_{b,2}$. For $u_0\in W_b^{1,2}$, one can see that
\[I_{\psi_R}(0)=\int\psi_R(x)|u_0|^2dx\leq\int_{|x|\geq\frac{R}{2}}|u_0|^2dx=o_R(1)\]
and
\[\int_{|x|> R}|u|^2dx\leq I_{\psi_R}(t).\]
Thus we conclude \eqref{d5}. The proof of Lemma \ref{lem8} is completed.\quad\quad\quad\quad $\square$

Next, we give the following important result, which shows that there exists no global solution whose $L^q(|x|^{\frac{c(q-2)}{p}}dx)$ norms are uniformly bounded in time.
\begin{lemma}\label{lem7}
Let $n\geq3,\ 2-n<b\leq0,\ V(x)\geq0$ and $x\cdot\nabla V\in L^{\frac{n}{2}}(|x|^{-\frac{nb}{2}}dx)$. Let $u\in C([0, T^\ast); W^{1,2}_b)$ be the corresponding solution to the INLS$_{b,V}$ with $u_0\in W_b^{1,2}$. If there exists $\delta>0$ such that the virial functional satisfies the bound
\begin{equation}\label{e16}
\sup_{t\in [0, T^\ast)}P(u)\leq-\delta,
\end{equation}
then there exists no global solution $u\in C([0, \infty); W^{1,2}_b)$ with
\begin{equation}\label{e6}
\sup_{t\in [0, \infty)}\|u(t)\|_{\frac{c(q-2)}{p},q}<\infty\quad\quad \mbox{for\ some}\ q>p+2.
\end{equation}
\end{lemma}

\begin{proof}
We assume by contradiction that there exists a global solution $u(t)$ to the INLS$_{b,V}$ such that
\begin{equation}\label{d3}
T^\ast=\infty\quad\quad \mbox{and}\quad\quad  C_0:=\sup_{t\in[0, \infty)}\|u(t)\|_{\frac{c(q-2)}{p},q}<\infty
\end{equation}
for some $q>p+2$. Then we claim that there exists a constant $0<C_1=C_1(C_0, E_b(u_0),$
$M(u_0))<\infty$ such that
\begin{equation}\label{g20}
C_1:=\sup_{t\in[0, \infty)}\|\nabla u(t)\|_{b,2}<\infty.
\end{equation}

To see it, we invoke the following weighted interpolation relationship
\begin{equation}\nonumber
\left(L^{p_0}(w_0^{p_0}dx),\ L^{p_1}(w_1^{p_1}dx)\right)_{\vartheta, \widetilde{q}}=L^{\widetilde{q}}(w_0^{(1-\vartheta)\widetilde{q}}w_1^{\vartheta\widetilde{q}}dx)
\end{equation}
cited in \cite{DF}, where $\frac{1}{\widetilde{q}}=\frac{1-\vartheta}{p_0}+\frac{\vartheta}{p_1}$ and $0<\vartheta<1$, which can be applied to give
\begin{equation}\label{e4}
\|u\|_{c,p+2}\leq C\|u\|^{1-\vartheta}_{\frac{c(q-2)}{p},q}\|u\|_2^{\vartheta}
\end{equation}
for all $q>p+2$. Since $V(x)\geq0$, then it follows from the energy conservation and \eqref{e4} that
\begin{eqnarray}\nonumber
\lefteqn{\|\nabla u\|_{b,2}\leq2E_{b,V}(u_0)+\frac{2}{p+2}\|u\|_{c,p+2}^{p+2}}\\\nonumber
&&\quad\quad\leq2E_{b,V}(u_0)+\frac{2}{p+2}\|u\|^{(1-\vartheta)(p+2)}_{\frac{c(q-2)}{p},q}\|u\|_2^{\vartheta(p+2)},
\end{eqnarray}
which together with \eqref{d3} yields the result \eqref{g20}.

Now we continue the argument of \eqref{e6}. Recalling the functional $I''_{\psi_R}(t)$ defined in \eqref{e13}, we claim that there exists $C=C(b, \textbf{p}_c, C_0, C_2)$ and $\vartheta_q>0$ such that the following nonradial version of the localized virial estimate
\begin{eqnarray}\nonumber
\lefteqn{I''_{\psi_R}(t)\leq 4(2-b)P(u)+C\|u\|_{L^2(|x|>R)}^{\vartheta_q(p+2)}}\\\label{e7}
&&\quad\quad\quad\ +\frac{C}{R^{2-b}}\|u\|_{L^2(|x|>R)}+C\|x\cdot\nabla V\|_{L^{\frac{n}{2}}(|x|>R;\ |x|^{-\frac{nb}{2}}dx)}
\end{eqnarray}
holds for $q>p+2$. 

Indeed, from \eqref{e13}, the estimates for $\mathcal{R}_1$ and $\mathcal{R}_2$ are the same as in \eqref{c5} and \eqref{d11} respectively, so it reduces to control the terms $\mathcal{R}_3$ and $\mathcal{R}_4$.

For  $\mathcal{R}_3$, invoking the interpolation inequality \eqref{e4} again, we have $\vartheta_q$ with $0<\vartheta_q=\frac{2(q-p-2)}{(q-2)(p+2)}<1$ such that
\begin{eqnarray}\nonumber
\lefteqn{\mathcal{R}_3\ \leq C_{p,c}\|u\|^{(1-\vartheta_q)(p+2)}_{L^q(|x|>R;\ |x|^{\frac{c(q-2)}{p}}dx)}\|u\|_{L^2(|x|>R)}^{\vartheta_q(p+2)}}\\\label{d9}
&&\leq C_{p,c}C_0^{(1-\vartheta_q)(p+2)}\|u\|_{L^2(|x|>R)}^{\vartheta_q(p+2)}.\quad\quad\quad\quad
\end{eqnarray}

For  $\mathcal{R}_4$, it follows from \eqref{e17} and \eqref{g20} that
\begin{equation}\label{e2}
\mathcal{R}_4\leq C_{b}C_1^2\|x\cdot\nabla V\|_{L^{\frac{n}{2}}(|x|>R;\ |x|^{-\frac{nb}{2}}dx)}.
\end{equation}
Thus, putting \eqref{c5}, \eqref{d11}, \eqref{d9}-\eqref{e2} all together, we conclude the estimate \eqref{e7}.

Finally, by \eqref{e16}, \eqref{e7} and Lemma \ref{lem8}, we have
\begin{eqnarray}\nonumber
\lefteqn{I''_{\psi_R}(s)\leq-4(2-b)\delta+o_R(1)+C\eta^{\frac{\vartheta_q(p+2)}{2}}+\frac{CM(u_0)^{\frac{1}{2}}}{R^{2-b}}+C\|x\cdot\nabla V\|_{L^{\frac{n}{2}}(|x|>R;\ |x|^{-\frac{nb}{2}}dx)}}\\\nonumber
&&\quad\ \leq-4(2-b)\delta+C\eta^{\frac{\vartheta_q(p+2)}{2}}+o_R(1)\quad\quad\quad\quad\quad\quad\quad\quad\quad\quad\quad\quad\quad\quad\quad\quad\quad\quad
\end{eqnarray}
for any $\eta>0, R>0$ and $s\in [0,\ \frac{\eta R^{1-\frac{b}{2}}}{C_1C_2\|u_0\|_2}]$.
We take $\eta=\eta_0>0$ sufficiently small such that
\[C\eta^{\frac{\vartheta_q(p+2)}{2}}\leq(4-2b)\delta,\]
which yields
\begin{equation}\label{d12}
I''_{\psi_R}(s)\leq-(4-2b)\delta+o_R(1).
\end{equation}

Now we set
\[T=T(R):=\alpha_0R^{1-\frac{b}{2}}=\frac{\eta_0R^{1-\frac{b}{2}}}{C_1C-2M(u_0)},\]
where $\alpha_0$ is independent of $R$. So, by integrating \eqref{d12} over $s\in[0,t]$ and then integrating over $t\in[0, T]$, one has
\begin{equation}\label{d15}
I_{\psi_R}(T)\leq I_{\psi_R}(0)+I'_{\psi_R}(0)\alpha_0R^{1-\frac{b}{2}}+\left(-(4-2b)\delta+o_R(1)\right)\alpha_0^2 R^{2-b}.
\end{equation}
To proceed it, we claim that for $R\rightarrow\infty$,
\begin{equation}\label{d14}
I_{\psi_R}(0)=o_R(1)R^{2-b},\quad\quad I'_{\psi_R}(0)=o_R(1)R^{1-\frac{b}{2}}.
\end{equation}

Indeed, by $\nabla\psi_R=\frac{\nabla\phi_R}{|x|^b}$, defined as \eqref{d13}, we deduce that
\begin{eqnarray}\nonumber
\lefteqn{I_{\psi_R}(0)\leq\frac{2}{2-b}\int_{|x|\leq\sqrt{R}}|x|^{2-b}|u_0|^2dx+\int_{\sqrt{R}\leq|x|\leq2R}(\int_0^x\frac{R^2\nabla\Theta(\frac{x}{R})}{|x|^b}dx)|u_0|^2dx}\\\nonumber
&&\quad\leq\frac{2}{2-b}R^{1-\frac{b}{2}}M(u_0)+CR^{2-b}\int_{|x|\geq\sqrt{R}}|u_0|^2dx=o_R(1)R^{2-b}.\quad\quad\quad\quad
\end{eqnarray}
On the other hand, we have
\begin{eqnarray}\nonumber
\lefteqn{I'_{\psi_R}(0)=4\textrm{Im}\int_{|x|\leq\sqrt{R}}x\cdot\nabla u_0\overline{u}_0dx+2\int_{\sqrt{R}\leq|x|\leq2R}R\partial_r\Theta(\frac{x}{R})\frac{x\cdot\nabla u_0}{r}\overline{u}_0dx}\\\nonumber
&&\quad\leq4\||x|^{1-\frac{b}{2}}\|_{L^\infty(|x|\leq\sqrt{R})}\|\nabla u_0\|_{b,2}\|u_0\|_2\\\nonumber
&&\quad\quad +2CR\||x|^{-\frac{b}{2}}\|_{L^\infty(\sqrt{R}\leq|x|\leq2R)}\|\nabla u_0\|_{b,2}\|u_0\|_{L^2(|x|\geq\sqrt{R})}\\\nonumber
&&\quad\leq4R^{\frac{2-b}{4}}\|\nabla u_0\|_{b,2}\|u_0\|_2+2CR^{1-\frac{b}{2}}\|\nabla u_0\|_{b,2}o_R(1)=o_R(1)R^{1-\frac{b}{2}},
\end{eqnarray}
which arrives at \eqref{d14}.

Hence, combining \eqref{d15} with \eqref{d14}, and choosing $R$ large enough, we infer that
\begin{equation}\label{g15}
I_{\psi_R}(T)\leq [-(4-2b)\delta\alpha_0^2+o_R(1)]R^{2-b}\leq -(2-b)\delta\alpha_0^2,
\end{equation}
which contradicts $I_{\psi_R}(T)=\int\psi_R(x)|u(x, T)|^2dx\geq0$.
Hence, we finish the proof of Lemma \ref{lem7}.

\end{proof}
With the above lemmas in hand, we give the proof of Theorem \ref{thmm1}.

\textbf{Proof of Theorem \ref{thmm1}.}
(i) Let $u_0\in \mathcal{N}^+$, it follows from Proposition \ref{lemm5} that the corresponding solution $u(t)\in \mathcal{N}^+$ for any $t\in[0, T^\ast)$. By the local well-posedness theory, we only need to control the $W_b^{1,2}$-norm of $u(t).$

To see it, we use Lemma \ref{lemma3} and $V\geq0$ to get
\[C_\omega\|u(t)\|^2_{W_b^{1,2}}\leq L_{\omega,V}(u(t))\leq\frac{2\textbf{p}_c}{\textbf{p}_c-4+2b}S_{\omega,V}(u(t))<Cm_{\omega,0}^{\alpha,\beta}\]
for all $t\in[0, T^\ast)$, where $C$ is a positive constant depending on $b$ and $\textbf{p}_c$. Thus we get $\|u(t)\|_{W_b^{1,2}}$ is bounded for any $t\in[0, T^\ast)$, and therefore $u(t)$ exists globally in $[0, \infty)$.

(ii) Let $u_0\in \mathcal{N}^-$, we obtain $u(t)\in \mathcal{N}^-$ by Proposition \ref{lemm5}. We now prove the blow-up results in two cases:

(1)  If $T^\ast<+\infty$, the local well-posedness theory implies that
$\lim_{t\rightarrow T^\ast}\|\nabla u(t)\|_{b,2}=\infty.$

(2) If $T^\ast=+\infty$. By Lemma \ref{lemma4}, there exists
\[\delta_1:=2(m_{\omega,0}^{n,2}-S_{\omega,V}(u_0))>0\]
such that
\[P(u(t))=\frac{1}{2-b}K_{\omega,V}^{n,2}(u(t))\leq-\delta_1\]
for all $t\in [0, T^\ast)$. Thus, Lemma \ref{lem7} implies that there exists a time sequence $\{t_n\}$ such that $t_n\rightarrow\infty$ and
\begin{equation}\label{e9}
\lim_{t_n\rightarrow \infty}\|u(t_n)\|_{\frac{c(q-2)}{p},q}=\infty
\end{equation}
for any $q>p+2$. For such $u(t_n)$, we use the Hardy-Sobolev inequality in Lemma \ref{lem11} to get
\begin{equation}\label{g26}
\|u(t_n)\|_{\frac{c(q-2)}{p},q}\leq C_b\|\nabla u(t_n)\|_{b,2}
\end{equation}
for all $2(2-b)<\textbf{p}_c\leq\frac{2c(2-b)}{b}$. Therefore, combining \eqref{e9} with \eqref{g26}, we obtain
\[\lim_{n\rightarrow \infty}\|\nabla u(t_n)\|_{b,2}=\infty.\]
The proof of Theorem \ref{thmm1} is completed.
 \quad\quad\quad $\square$

Finally, we give the proof of Corollary \ref{thm4}.

\textbf{Proof of Corollary \ref{thm4}.}  From Theorem \ref{thmm1} (ii), it is sufficient to show that \eqref{g40} implies that
$u_0\in\mathcal{N}^-.$

Firstly, we claim that the following two conditions are equivalent:
\begin{equation}\label{k4}
E_{b,V}(u_0)M(u_0)^\sigma<E_{b}(Q_{1})M(Q_{1})^\sigma\quad \Leftrightarrow\quad S_{\omega,V}(u_0)<m_{\omega,0}.
\end{equation}

Indeed, by Proposition \ref{prop1}, we obtain $m_{\omega,0}=S_{\omega,0}(Q_{\omega})$,
where $Q_{\omega}(x)$ is the ground state of the elliptic equation \eqref{c1}. So for $\omega>0$, we write
\[f(\omega):=S_{\omega,0}(Q_{\omega})-S_{\omega,V}(u_0).\]
Let $Q_{\omega}=\omega^{\frac{2-b+c}{(2-b)p}}Q_{1}(\omega^{\frac{1}{2-b}}x)$, one can easily check that $Q_{1}(x)$ satisfies the equation \eqref{c1} with $\omega=1$.
This implies that
\[S_{\omega,0}(Q_{\omega})=\omega^{\frac{(2-b)(p+2)-\textbf{p}_c}{(2-b)p}}S_{1,0}(Q_{1}).\]
By rewriting
\begin{equation}\label{g10}
f(\omega)=\omega^{\frac{(2-b)(p+2)-\textbf{p}_c}{(2-b)p}}S_{1,0}(Q_{1})-S_{\omega,V}(u_0),
\end{equation}
we conclude that the claim \eqref{k4} is valid if and only if $\sup_{\omega>0}f(\omega)>0$.

Since
\[f'(\omega)=\frac{(2-b)(p+2)-\textbf{p}_c}{(2-b)p}\omega^{\frac{2(2-b)-\textbf{p}_c}{(2-b)p}}S_{1,0}(Q_{1})-\frac{1}{2}M(u_0),\]
we solve $f'(\omega_0)=0$ and obtain
\[\omega_0=\left(\frac{(2-b)p}{2(2-b)(p+2)-2\textbf{p}_c}\frac{M(u_0)}{S_{1,0}(Q_{1})}\right)^{\frac{(2-b)p}{2(2-b)-\textbf{p}_c}}>0.\]
A simple calculation shows that the function $f$ has a maximum value at $\omega=\omega_0$.
Therefore, the statement $S_{\omega,V}(u_0)<m_{\omega,0}$ holds if and only if $f(\omega_0)>0$.

Substituting $\omega_0$ into \eqref{g10}, we have
\begin{eqnarray}\nonumber
\lefteqn{f(\omega_0)=\left(\frac{2(2-b)(p+2)-2\textbf{p}_c}{(2-b)p}\right)^{\frac{(2-b)p}{\textbf{p}_c-4+2b}}\frac{\textbf{p}_c-4+2b}{2(2-b)(p+2)-2\textbf{p}_c}}\\\nonumber
&&\quad\quad\times\frac{S_{1,0}(Q_{1})^{\frac{(2-b)p}{\textbf{p}_c-4+2b}}}{M(u_0)^{\frac{(2-b)(p+2)-\textbf{p}_c}{\textbf{p}_c-4+2b}}}-E_{b,V}(u_0).\quad\quad\quad\quad\quad\quad\quad\quad\quad\quad
\end{eqnarray}
Thus it follows from $f(\omega_0)>0$ that
\begin{eqnarray}\nonumber
\lefteqn{\left(\frac{2(2-b)(p+2)-2\textbf{p}_c}{(2-b)p}\right)^{\frac{(2-b)p}{\textbf{p}_c-4+2b}}\frac{\textbf{p}_c-4+2b}{2(2-b)(p+2)-2\textbf{p}_c}S_{1,0}(Q_{1})^{\frac{(2-b)p}{\textbf{p}_c+2b-4}}}\\\label{g11}
&&\quad\quad\quad\quad\quad\quad\quad\quad\quad>E_{b,V}(u_0)M(u_0)^\sigma,\quad\quad\quad\quad\quad\quad\quad\quad\quad\quad\quad\quad\quad\quad
\end{eqnarray}
where $\sigma$ is defined in \eqref{k6}.

Recalling the Pohozaev identities \eqref{a3} related to $Q_{1}$, we deduce that
\[S_{1,0}(Q_{1})=\frac{(2-b)p}{2\textbf{p}_c}\|\nabla Q_{1}\|^2_{b,2}\]
and
\[E_{b}(Q_{1})M(Q_{1})^\sigma=\left(\frac{(2-b)(p+2)}{\textbf{p}_c}-1\right)^{\frac{(2-b)(p+2)-\textbf{p}_c}{\textbf{p}_c-4+2b}}\frac{\textbf{p}_c-4+2b}{2\textbf{p}_c}\|\nabla Q_{1}\|_{b,2}^{\frac{2(2-b)p}{\textbf{p}_c-4+2b}},\]
which yields that
\[\mbox{L.H.S.\ of}\ \eqref{g11}= E_{b}(Q_{1})M(Q_{1})^\sigma.\]
Hence, we obtain $E_{b}(Q_{1})M(Q_{1})^\sigma> E_{b,V}(u_0)M(u_0)^\sigma.$

Next, we need to show that \eqref{g40} implies that
 \begin{equation}\label{g41}
K_{\omega,V}^{n,2}(u_0)<0.
\end{equation}

To this end, by the definition of the functional $P$ and \eqref{k1}, we get
\begin{eqnarray}\nonumber
\lefteqn{P(u_0)\leq\|\nabla u_0\|^2_{b,2}+\int V(x)|u_0|^2dx-\frac{\textbf{p}_c}{(p+2)(2-b)}\|u_0\|^{p+2}_{c,p+2}}\\\nonumber
&&\quad=\frac{\textbf{p}_c}{2-b}E_{b,V}(u_0)-\left(\frac{\textbf{p}_c}{4-2b}-1\right)\|u_0\|^2_{\dot{H}_{b, V}^1}\\\nonumber
&&\quad<\frac{\textbf{p}_c}{2-b}\frac{E_{b}(Q_{1})M(Q_{1})^\sigma}{M(u_0)^\sigma}-\left(\frac{\textbf{p}_c}{4-2b}-1\right)\frac{\|\nabla Q_{1}\|^2_{b,2}\|Q_{1}\|_2^{2\sigma}}{\|u_0\|_2^{2\sigma}},
\end{eqnarray}

so it follows from the Pohozaev's identities and \eqref{g40} that $P(u_0)<0,$ and then \eqref{g41} is proved.

Finally, combining \eqref{k4} with \eqref{g41}, we conclude that $u_0\in\mathcal{N}^-$.  Corollary \ref{thm4} is proved.\quad\quad\quad\quad\quad\quad $\square$

\vspace{4mm}

\subsection*{Acknowledgments}
This work is supported by the National Natural Science Foundation of China (Grant No. 12401147).

\subsection*{Competing Interests}
The authors have no conflicts to disclose.


\begin{thebibliography}{60}


\bibitem{AKJ} {\sc J. An, J. Kim, R. Jang}, {\em Global existence and blow-up for the focusing inhomogeneous nonlinear Schr\"{o}dinger equation with inverse-square potential}, Discrete Contin. Dyn. Syst., 28(2) (2023), pp.~1046--1067.


\bibitem{CKN} {\sc L. Caffarelli, R. Kohn, L. Nirenberg}, {\em  First order interpolation inequalities with weights}, Compos. Math., 53 (1984), pp.~259--275.
\bibitem{CH} {\sc D. Cao, P. Han}, {\em Inhomogeneous critical nonlinear Schr\"{o}dinger equations with a harmonic potential}, J. Math. Phys., 51 (2010), pp.~043505.

\bibitem{TC} {\sc T. Cazenave}, {\em Semilinear Schr\"{o}dinger Equations}, Courant Lecture Notes in Mathematics,  vol.10, New York University, Courant Institute of Mathematical Sciences, New York; American Mathematical Society, Providence, RI, 2003.
\bibitem{DLY} {\sc Y. Deng, Y. Li, F. Yang}, {\em On the positive radial solutions of a class of singular semilinear elliptic equations}, J. Differ. Equations, 253 (2012), pp.~481--501.
\bibitem{VDD2} {\sc V.D. Dinh}, {\em Global dynamics for a class of inhomogeneous nonlinear Schr\"{o}dinger equations with potential}, Math. Nachr., 294(4) (2021), pp.~672--716.
\bibitem{D5} {\sc V.D. Dinh}, {\em Blowup of $H^1$ solutions for a class of the focusing inhomogeneous nonlinear Schr\"{o}dinger equation}, Nonlinear Anal., 174 (2018), pp.~169--188.
\bibitem{DK} {\sc V.D. Dinh, S. Keraani}, {\em Long time dynamics of non-radial solutions to inhomogeneous nonlinear Schr\"{o}dinger equations}, SIAM J. Math. Anal., 53(4) (2021), pp.~4765--4811.
\bibitem{DWZ} {\sc D. Du, Y. Wu, K. Zhang}, {\em On blow-up criterion for the nonlinear Schr\"{o}dinger equations}, Discrete Contin. Dyn. Syst., 36 (2016), pp.~3639-3650.
\bibitem{F4} {\sc L.G. Farah}, {\em Global well-posedness and blow-up on the energy space for the inhomogeneous nonlinear Schr\"{o}dinger equation}, J. Evol. Equ., 16(1) (2016), pp.~193--208.
\bibitem{DF} {\sc D. Freitag}, {\em Real interpolation of weighted $L^p$-spaces}, Math. Nachr., 86 (1978), pp.~15--18.
\bibitem{G5} {\sc F. Genoud}, {\em An inhomogeneous, $L^2$-critical, nonlinear Schr\"{o}dinger equation}, Z. Anal. Anwend., 31(3) (2012), pp.~283--290.
\bibitem{GS4} {\sc F. Genoud, C.A. Stuart}, {\em Schr\"{o}dinger equations with a spatially decaying nonlinearity: existence and stability of standing waves}, Discrete Contin. Dyn. Syst., 21(1) (2008), pp.~137--186.
\bibitem{GC4} {\sc C.M. Guzman}, {\em On well posedness for the inhomogeneous nonlinear Schr\"{o}dinger equation}, Nonlinear Anal. Real World Appl., 37 (2017), pp.~249--286.
\bibitem{GG} {\sc R.T. Glassey}, {\em On the blowing up of solutions to the Cauchy problem for nonlinear Schr\"{o}dinger equations}, J. Math. Phys., 18 (1976), pp.~1794--1797.

\bibitem{G3} {\sc Z. Guo, X. Guan, F. Wan}, {\em Existence and regularity of positive solutions of a degenerate elliptic problem}, Math. Nachr., 292 (2019), pp.~56--78.
\bibitem{GGW} {\sc Z. Guo, X. Guan, F. Wan}, {\em Sobolev type embedding and weak solutions with a prescribed singular set}, Sci. China Math., 59(10) (2016), pp.~1975--1994.
\bibitem{G4} {\sc Z. Guo, Y. Li, F. Wan}, {\em Asymptotic behavior at the isolated singularities of solutions of some equations on singular manifolds with conical metrics}, Commun. Part. Diff. Eq., 45 (2020), pp.~1647--1681.

\bibitem{GWY} {\sc Q. Guo, H. Wang, X. Yao}, {\em Scattering and blow-up criteria for 3d cubic focusing nonlinear inhomogeneous NLS with a potential}, arXiv.org/abs/1801.05165.
\bibitem{GI} {\sc S. Gustafson, T. Inui}, {\em Scattering and blow-up for threshold even solutions to the nonlinear Schr\"{o}dinger equation with repulsive delta potential at low frequencies}, Journal of Differential Equations, 412 (2024), pp.~758-796.
\bibitem{HI} {\sc M. Hamano, M. Ikeda}, {\em Global dynamics below the ground state for the focusing Schr\"{o}dinger equation with a potential}, J. Evol. Equ., 20(3) (2020), pp.~1131--1172.
\bibitem{H1} {\sc Y. Hong}, {\em Scattering for a nonlinear Schr\"{o}dinger equation with a potential}, Commun. Pure Appl. Anal., 15 (2016), pp.~1571--1601.
\bibitem{IMN} {\sc S. Ibrahim, N. Masmoudi, K. Nakanishi}, {\em Scattering threshold for the focusing nonlinear Klein-Gordon equation}, Anal. PDE, 4 (2011), pp.~405--460.
\bibitem{II} {\sc M. Ikeda, T. Inui}, {\em Global dynamics below the standing waves for the focusing semilinear Schr\"{o}dinger equation with a repulsive Dirac delta potential}, Anal. PDE, 10(2) (2017), pp.~481--512.

\bibitem{KMVZ} {\sc R. Killip, J. Murphy, M. Visan, J. Zheng}, {\em The focusing cubic NLS with inverse-square potential in three space dimensions}, Differ. Integral Equ., 30 (2017), pp.~161--206.

\bibitem{M} {\sc R. Musina}, {\em Ground state solutions of a critical problem invovling cylindrical weights}, Nonlinear Anal., 68 (2008), pp.~3972--3986.

\bibitem{SC} {\sc Y. Shen, Z. Chen}, {\em Nonlinear degenerate elliptic equation with Hardy potential and critical parameter}, Nonlinear Anal., 69 (2008), pp.~1462--1477.
\bibitem{SW2} {\sc N. Shioji, K. Watanabe}, {\em Uniqueness and nondegeneracy of positive radial solutions of $div(\rho\nabla u)+\rho(-gu+hu^p)=0$}, Calculus Var. Partial Differ. Equations, 55 (2016), pp.~32.

\bibitem{WW} {\sc Z. Wang, M. Willem}, {\em Caffarelli-Kohn-Nirenberg inequalities with remainder terms, J. Funct. Anal., 203} (2003), pp.~550--568.
\bibitem{ZYS} {\sc Y. Zhang, J. Yang, Y. Shen}, {\em Solutions for nonlinear elliptic equations with general weight in the Sobolev-Hardy space}, P. Am. Math. Soc., 139(1) (2011), pp.~219--230.

\bibitem{ZOZ} {\sc B. Zheng, T. Ozawa, J. Zhai}, {\em Blow-up solutions for a class of divergence Schr\"{o}dinger equations with intercritical inhomogeneous nonlinearity}, J. Math. Phys., 64 (2023), pp.~011502.
\bibitem{ZO} {\sc B. Zheng, T. Ozawa}, {\em The blow-up dynamics for the divergence Schr\"{o}dinger equations with inhomogeneous nonlinearity}, preprint.
\bibitem{ZZ1} {\sc B. Zheng, W. Zhu}, {\em Strong instability of standing waves for the divergence Schr\"{o}dinger equation with inhomogeneous nonlinearity}, J. Math. Anal. Appl., 530 (2024), pp.~127730.





\end{thebibliography}
\end{document}